\documentclass{article}

\usepackage{palatino}
\usepackage{amssymb,amsmath}
\usepackage{amscd}
\usepackage{gastex}
\usepackage[usenames]{color}
\usepackage{amsthm}
\usepackage{graphics}
\usepackage{subfigure}
\newtheorem{lemma}{Lemma}
\newtheorem{theorem}[lemma]{Theorem}
\newtheorem{corollary}[lemma]{Corollary}

\newtheorem{example}[lemma]{Example}
\newtheorem{definition}[lemma]{Definition}
\newtheorem{remark}[lemma]{Remark}
\usepackage{longtable}

\newcommand{\casos}[4]{
\left\{
\begin{array}{ll}
#1 & \mbox{ #2 }\\
#3 & \mbox{ #4 }\\
\end{array}
\right.
}
\newcommand{\trescasos}[6]{
\left\{
\begin{array}{ll}
#1 & \mbox{ #2 }\\
#3 & \mbox{ #4 }\\
#5 & \mbox{ #6 }\\
\end{array}
\right.
}
\newcommand{\quatrecasos}[8]{
\left\{
\begin{array}{ll}
#1 & \mbox{ #2 }\\
#3 & \mbox{ #4 }\\
#5 & \mbox{ #6 }\\
#7 & \mbox{ #8 }\\
\end{array}
\right.
}

\renewcommand\leq{\leqslant}
\renewcommand\geq{\geqslant}
\renewcommand\preceq{\preccurlyeq}

\newcommand\pleq{\preceq}

\title{Numerical Semigroups and Codes}

\date{{\small Chapter 5 of Algebraic Geometry Modeling in Information Theory, \\E. Martinez-Moro (ed.), World Scientific, 2013.}}

\author{Maria Bras-Amor\'os}
\begin{document}

\maketitle

\begin{abstract}
A numerical semigroup is a subset of ${\mathbb N}$ containing 0, closed under addition and with finite complement in ${\mathbb N}$. An important example of numerical semigroup is given by the Weierstrass semigroup at one point of a curve.
In the theory of algebraic geometry codes, Weierstrass semigroups are crucial
for defining bounds on the minimum distance as well as for defining improvements on the dimension of codes. We present these applications and some
theoretical problems related to classification, characterization and counting of
numerical semigroups.
\end{abstract}

\tableofcontents

\section*{Introduction}

Numerical semigroups are probably one of the most simple mathematical objects. However they are involved in very hard (and some very old) problems.
They can also be found in several applied fields such as 
error-correcting codes, cryptography, or combinatorial structures for privacy applications.

In the present chapter we present numerical semigroups with some of 
the related classical problems and we explore their 
importance in the field of algebraic-geometry codes.

The material is divided into two parts. In the first part 
we give a brief introduction to Weierstrass semigroups as 
the paradigmatic example of numerical semigroups, we present some classical problems related to general numerical semigroups, we deal with 
some problems on classification and characterization of numerical semigroups
which have an application to coding theory, and we finally present a conjecture on counting numerical semigroups by their genus.

In the second part we present
one-point algebraic-geometry codes and we 
focus on the applications that numerical semigroups
have for defining bounds on the minimum distance as well as for defining improvements on the dimension of these codes. 
Based on the decoding algorithm for one-point codes
one can deduce sufficient conditions for decoding, and
from these conditions one can define minimal sets of parity checks 
(and so codes with improved correction capability)
either for correcting any kind of error or 
at least for guaranteeing the correction of the so-called generic errors.
The decoding conditions are related to the associated Weierstrass semigroups and so
the improvements can be defined in terms of semigroups.

\section{Numerical semigroups}

\newcommand\alert[1]{{\em #1}\index{#1}}
%{{\boldmath{\bf #1}}}

\subsection{Paradigmatic example: Weierstrass semigroups on algebraic curves}

\nocite{Pretzel,HoLiPe:agc,Stichtenoth:AFFaC,Fulton,Torres:notes,Giulietti:notes}

\subsubsection{Algebraic curves}

%Let ${\mathbb N}_0$ denote the set of all non-negative integers.

%\paragraph{Affine curves}
Consider a field $K$ and a bivariate polynomial $f(x,y)\in K[x,y]$.
If $\bar K$ is the algebraic closure of $K$, the (plane)
\alert{affine curve} associated to $f$ is the set of points in $\bar K^2$
at which $f$ vanishes. 
%
%\paragraph{Projective curves}
Now given a {\it homogeneous} polynomial $F(X,Y,Z)\in K[X,Y,Z]$
the (plane)
\alert{projective curve} associated to $F$ is the set of points in 
${\mathbb P}^2(\bar K)$
at which $F$ vanishes. 
We use the notation \alert{${\mathcal X}_F$} to denote it.

%\paragraph{Homogenization and dehomogenization}
From the affine curve 
defined by the polynomial $f(x,y)$ 
of degree $d$ 
we can obtain a projective curve 
defined by the \alert{homogenization} of $f$, that is, $f^*(X,Y,Z)=Z^df(\frac{X}{Z},\frac{Y}{Z})$.
Conversely, a projective curve defined by a homogeneous polynomial 
$F(X,Y,Z)$ defines three affine curves with 
\alert{dehomogenized}
polynomials $F(x,y,1)$, $F(1,u,v)$, $F(w,1,z)$.
The points $(a,b)\in \bar K^2$ of the affine curve defined by
$f(x,y)$ correspond to the 
points $(a:b:1)\in {\mathbb P}^2(\bar K)$ of ${\mathcal X}_{f^*}$.
Conversely, the points $(X:Y:Z)$ 
with $Z\neq 0$ (resp. $X\neq 0$, $Y\neq 0$) 
of a projective curve ${\mathcal X}_F$
correspond to the points of the 
affine curve defined by $F(x,y,1)$ 
(resp. $F(1,u,v)$, $F(w,1,z)$)
and so they are called affine points of $F(x,y,1)$.
The points with $Z=0$ are said to be \alert{at infinity}. 

In the case $K={\mathbb F}_q$, any point of ${\mathcal X}_F$ is in 
${\mathbb P}^2({\mathbb F}_{q^m})$
for some $m$.
If $L/K$ is a field extension we define the \alert{$L$-rational points} of ${\mathcal X}$ as the points in the set ${\mathcal X}_F(L)={\mathcal X}_F\cap L^2$.

We will assume that $F$ is irreducible in any field extension of $K$ (i.e. \alert{absolutely irreducible}). 
Otherwise the curve is a proper union of two curves.

%The \alert{degree} of ${\mathcal X}_F$ is the degree of $F$.

%Let $(F)$ be the ideal generated by $F$.
%Since $F$ is irreducible and $K(X,Y,Z)$ is a principal ideal domain, then $F$ is prime.
If two polynomials in $K(X,Y,Z)$ differ by a multiple of $F$, when evaluating them at a point of ${\mathcal X}_F$ we obtain the same value.
Thus it makes sense to consider $K(X,Y,Z)/(F)$.
Since $F$ is irreducible, $K(X,Y,Z)/(F)$ 
is an integral domain and we can construct its field of fractions $Q_F$.
For evaluating one such fraction at a projective point we want the result not to depend on the representative of the projective point. Hence, we require the numerator and the denominator to have one representative each, which is a homogeneous polynomial and both having the same degree. 
The \alert{function field} of ${\mathcal X}_F$,
denoted \alert{$K(\mathcal X_F)$}, is the set of elements of $Q_F$ admitting one such representation.
Its elements are the
\alert{rational functions} of ${\mathcal X}_F$. 
We say that a rational function $f\in K({\mathcal X}_F)$
is \alert{regular in a point} $P$ if 
there exists a representation of it as a fraction 
$\frac{G(X,Y,Z)}{H(X,Y,Z)}$
with $H(P)\neq 0$. In this case we
define $f(P)=\frac{G(P)}{H(P)}$.
The ring of all rational functions regular in $P$ is denoted 
\alert{${\mathcal O}_P$}. Again it is an integral domain and this time its field of fractions is $K({\mathcal X_F})$.

Let $P\in {\mathcal X}_F$ be a point.
If all the partial derivatives $F_X, F_Y, F_Z$ vanish at $P$
then $P$ is said to be a \alert{singular point}.
Otherwise it is said to be a \alert{simple point}.
Curves without singular points are called \alert{non-singular}, \alert{regular} or \alert{smooth} curves.
%The \alert{tangent line} at a singular point $P$ of ${\mathcal X}_F$ is defined by the equation $F_X(P)X+F_Y(P)Y+F_Z(P)Z=0.$

From now on we will assume that $F$ is absolutely irreducible and that 
${\mathcal X}_F$ is smooth.

The \alert{genus} of a smooth plane 
curve ${\mathcal X}_F$  may be defined as
$$g=\frac{(\deg(F)-1)(\deg(F)-2)}{2}.$$ For general curves the genus 
is defined using differentials on a curve which is out of the purposes of this survey.

\subsubsection{Weierstrass semigroup}

\begin{theorem}
Consider a point $P$ in the projective curve ${\mathcal X}_F$.
There exists 
$t\in{\mathcal O}_P$
such that for any non-zero $f\in K({\mathcal X_F})$ there exists 
a unique integer $v_P(f)$ with $$f=t^{v_P(f)}u$$ for some $u\in{\mathcal O}_P$ with $u(P)\neq 0$.
The value $v_P(f)$ depends only on ${\mathcal X}_F$, $P$.

If $G(X,Y,Z)$ and $H(X,Y,Z)$ are two homogeneous polynomials 
of degree $1$ such that $G(P)=0$, $H(P)\neq 0$, 
and $G$ is not a constant multiple of $F_X(P)X+F_Y(P)Y+F_Z(P)Z$,
then we can take $t$ to be the class in ${\mathcal O}_P$ 
of $\frac{G(X,Y,Z)}{H(X,Y,Z)}$.
\end{theorem}

An element such as $t$ is called a \alert{local parameter}.
If there is no confusion we will write 
$\frac{G(X,Y,Z)}{H(X,Y,Z)}$ for its class in ${\mathcal O}_P$.
The value $v_P(f)$ is called the \alert{valuation} of $f$ at $P$.
The point $P$ is said to be a \alert{zero} of multiplicity $m$ if $v_P(f)=m>0$ 
and a \alert{pole} of multiplicity $-m$ if $v_P(f)=m<0$.
The valuation satisfies that 
$v_P(f)\geq 0$ if and only if $f\in{\mathcal O}_P$
and that in this case $v_P(f)> 0$ if and only if $f(P)=0$.
\begin{lemma}
\label{lemma:propsv}
\begin{enumerate}
\item $v_P(f)=\infty$ if and only if $f=0$
\item $v_P(\lambda f)=v_P(f)$ for all non-zero $\lambda\in K$
\item $v_P(f g)=v_P(f)+v_P(g)$
\item $v_P(f + g)\geq \min\{v_P(f),v_P(g)\}$ and equality holds if 
$v_P(f)\neq v_P(g)$
\item If $v_P(f)= v_P(g)\geq 0$ then there exists $\lambda\in K$ such that
$v_P(f-\lambda g)>v_P(f)$.
\end{enumerate}
\end{lemma}

%\begin{lemma}
%A non-zero $f\in K({\mathcal X}_F)$ satisfies $v_P(f)=0$ for all points $P\in{\%mathcal X}_F$ except for a finite number and $\sum_{P\in{\mathcal X}_F}v_P(f)=0%$.
%Moreover, if $v_P(f)\geq 0$ for all $P\in{\mathcal X}_F$ 
%then $f$ is a non-zero constant
%and in this case $v_P(f)= 0$ for all $P\in{\mathcal X}_F$ 
%\end{lemma}

Let \alert{$L(m P)$} be the set of rational functions having only poles at 
$P$ and with pole order at most $m$. It is a $K$-vector space and so we can define \alert{$l(mP)$}$=\dim_K(L(m P))$.
One can prove that $l(m P)$ is either $l((m-1) P)$ or $l((m-1) P)+1$.
There exists a rational function $f\in K({\mathcal X}_F)$
having only one pole at $P$
with $v_P(f)=-m$ if and only if $l(m P)=l((m-1) P)+1$.

Let $A=\bigcup_{m\geq 0}{L(m P)}$, 
that is, $A$ is 
the ring of rational functions having poles only at $P$. 
Define $\Lambda=\{-v_P(f): f\in A\setminus\{0\}\}.$
%=\{-v_i: i\in{\mathbb N}_0\}$$ with $-v_i<-v_{i+1}$.
It is obvious that $\Lambda\subseteq{\mathbb N}_0$,
where ${\mathbb N}_0$ denotes the set of all non-negative integers.

\begin{lemma}
The set $\Lambda\subseteq{\mathbb N}_0$ satisfies
\begin{enumerate}
\item $0\in \Lambda$
\item $m+m'\in \Lambda$ whenever $m,m'\in \Lambda$
\item ${\mathbb N}_0\setminus\Lambda$ has a finite number of elements
\end{enumerate}
\end{lemma}
\begin{proof}
\begin{enumerate}
\item Constant functions $f=a$
have no poles and satisfy $v_P(a)=0$ for all $P\in{\mathcal X}_F$. 
Hence, $0\in \Lambda$. 
\item 
If $m,m'\in \Lambda$ then there exist $f,g\in A$
with $v_P(f)=-m$, $v_P(g)=-m'$. Now,
by Lemma~\ref{lemma:propsv},
$v_P(fg)=-(m+m')$
and so $m+m'\in \Lambda$.
\item
The well-known Riemann-Roch theorem implies that $l(m P)=m+1-g$ if $m\geq 2g-1$.
On one hand this means that $m\in\Lambda$ for all $m\geq 2g$, 
and on the other hand, this means that $l(m P)=l((m-1) P)$ only for $g$ different values of $m$. So, the number of elements in ${\mathbb N}_0$
which are not in $\Lambda$ is equal to the genus.
\end{enumerate}
\end{proof}
The three properties of a subset of ${\mathbb N}_0$ in the previous
lemma will constitute the definition of a \alert{numerical semigroup}.
The particular numerical semigroup of the lemma is 
called the
\alert{Weierstrass semigroup} at $P$ and the elements in 
${\mathbb N}_0\setminus\Lambda$ are called the \alert{Weierstrass gaps}.

\subsubsection{Examples}

\begin{example}[Hermitian curve]
\label{example:hermite}
Let $q$ be a prime power.
The Hermitian curve ${\mathcal H}_q$ over ${\mathbb F}_{q^2}$
is defined by the affine equation
$x^{q+1}=y^{q}+y$ and homogeneous equation
$X^{q+1}-Y^qZ-YZ^q=0$.
It is easy to see that
its partial derivatives are $F_X=X^q$,
$F_Y=-Z^q$, $F_Z=-Y^q$ and so there is no projective point at which
${\mathcal H}_q$ 
is singular.
The point $P_\infty=(0:1:0)$
is the unique point of ${\mathcal H}_q$ at infinity.

We have $F_X(P_\infty)X+F_Y(P_\infty)Y+F_Z(P_\infty)Z=-Z$ and so $t=\frac{X}{Y}$ 
is a local parameter at $P_\infty$.
The rational functions $\frac{X}{Z}$ and $\frac{Y}{Z}$
are regular everywhere except at $P_\infty$. So, they belong to $\cup_{m\geq 0}L(m P_\infty)$.
One can derive from the homogeneous equation of the curve that 
$t^{q+1}=\left(\frac{Z}{Y}\right)^q+\frac{Z}{Y}$.
So, $v_{P_\infty}(\left(\frac{Z}{Y}\right)^q+\frac{Z}{Y})=q+1$. 
By Lemma~\ref{lemma:propsv} one can deduce that
$v_{P_\infty}(\frac{Z}{Y})=q+1$ 
and so $v_{P_\infty}(\frac{Y}{Z})=-(q+1)$. 
On the other hand, since 
$\left(\frac{X}{Z}\right)^{q+1}=\left(\frac{Y}{Z}\right)^q+\frac{Y}{Z}$,
we have $(q+1)v_{P_\infty}(\frac{X}{Z})=-q(q+1)$. So, 
$v_{P_\infty}(\frac{X}{Z})=-q$.

We have seen that $q,q+1\in\Lambda$.
In this case $\Lambda$ contains what we will call later the
semigroup generated by $q,q+1$ whose complement in ${\mathbb N}_0$ has $\frac{q(q-1)}{2}=g$ elements.
Since we know that the complement of $\Lambda$ in 
${\mathbb N}_0$ also has $g$ elements, this means that both semigroups are the same.

For further details on the Hermitian curve see \cite{Stichtenoth:hermite,HoLiPe:agc}.
\end{example}

\begin{example}[Klein quartic]
\label{example:klein}
The \alert{Klein quartic}
%\index{Klein quartic!curve}
over ${\mathbb F}_q$
is defined by the affine equation
$x^3y+y^3+x=0.$
We shall see that if $\gcd(q,7)=1$
then ${\mathcal K}$ is smooth.
Its defining homogeneous polynomial is $F=X^3Y+Y^3Z+Z^3X$
and its partial derivatives are
$F_X=3X^2Y+Z^3$,
$F_Y=3Y^2Z+X^3$,
$F_Z=3Z^2X+Y^3$.
If the characteristic of ${\mathbb F}_{q^2}$ is $3$ then 
$F_X=F_Y=F_Z=0$ implies 
$X^3=Y^3=Z^3=0$ and so $X=Y=Z=0$.
Hence there is no projective point $P=(X:Y:Z)$ at which ${\mathcal K}$ 
is singular.
Otherwise, 
if the characteristic of ${\mathbb F}_{q^2}$ is different than $3$ then 
$F_X=F_Y=F_Z=0$ implies $X^3Y=-3Y^3Z$ and $Z^3X=-3X^3Y=9Y^3Z$.
Now the equation of the curve translates to $-3Y^3Z+Y^3Z+9Y^3Z=7Y^3Z=0$.
By hypothesis $\gcd(q,7)=1$ and so either $Y=0$ or $Z=0$.
In the first case, $F_Y=0$ implies $X=0$ and $F_X=0$ implies $Z=0$, a contradiction, and in the second case, $F_Y=0$ implies $X=0$ and $F_Z=0$ implies $Y=0$, 
another contradiction.

Let $P_0=(0:0:1)$. One can easily check that $P_0\in{\mathcal K}$.
We have $F_X(P_0)X+F_Y(P_0)Y+F_Z(P_0)Z=X$ and so $t=\frac{Y}{Z}$ 
is a local parameter at $P_0$.
From the equation of the curve we get 
$\left(\frac{X}{Y}\right)^3+\frac{Z}{Y}+\left(\frac{Z}{Y}\right)^3\frac{X}{Y}=0.$
So, at least one of the next equalities holds
\begin{itemize}
\item $3v_{P_0}(\frac{X}{Y})=v_{P_0}(\frac{Z}{Y})$
\item $3v_{P_0}(\frac{X}{Y})=3v_{P_0}(\frac{Z}{Y})+v_{P_0}(\frac{X}{Y})$
\item $v_{P_0}(\frac{Z}{Y})=3v_{P_0}(\frac{Z}{Y})+v_{P_0}(\frac{X}{Y})$
\end{itemize}
Since $t=\frac{Y}{Z}$ 
is a local parameter at $P_0$, $v_{P_0}(\frac{Z}{Y})=-1$.
Now, since $v_{P_0}(\frac{X}{Y})$ is an integer, only the third equality is possible, which leads to the conclusion that $v_{P_0}(\frac{X}{Y})=2$.
Similarly,
$\left(\frac{X}{Z}\right)^3\frac{Y}{Z}+\left(\frac{Y}{Z}\right)^3+\frac{X}{Z}=0$ gives that at least one of the next equalities holds
\begin{itemize}
\item $3v_{P_0}(\frac{X}{Z})+v_{P_0}(\frac{Y}{Z})=3v_{P_0}(\frac{Y}{Z})$
\item $3v_{P_0}(\frac{X}{Z})+v_{P_0}(\frac{Y}{Z})=v_{P_0}(\frac{X}{Z})$
\item $3v_{P_0}(\frac{Y}{Z})=v_{P_0}(\frac{X}{Z})$
\end{itemize}
Again only the 
third equality is possible and this leads to $v_{P_0}(\frac{X}{Z})=3$.

Now we consider the rational functions
$f_{ij}=\frac{Y^iZ^j}{X^{i+j}}$.
We have already seen that $v_{P_0}(f_{ij})=-2i-3j$ and we want to see under which conditions 
$f_{ij}\in\cup_{m\geq 0}L(m P_0)$.
This is equivalent to see when it has no poles rather than $P_0$.
The poles of $f_{ij}$ may only be at points with $X=0$ and so only at
$P_0$ and $P_1=(0:1:0)$.
Using the symmetries of the curve we get
%$v_{P_1}(\frac{Z}{Y})=3$, 
%$v_{P_1}(\frac{X}{Y})=1$, 
$v_{P_1}(\frac{Y}{X})=-1$, 
$v_{P_1}(\frac{Z}{X})=2$.
So, $v_{P_1}(f_{ij})=-i+2j$.
Then 
$f_{ij}\in\cup_{m\geq 0}L(m P_0)$ if and only if 
$-i+2j\geq 0$. 
We get that $\Lambda$ contains 
$\{2i+3j:i,j\geq 0, 2j\geq i\}=\{0,3,5,6,7,8,\dots\}$. 
This has $3$ gaps which is exactly the genus of ${\mathcal K}$.
So, $$\Lambda=\{0,3,5,6,7,8,9,10,\dots\}.$$ 

It is left as an exercise to prove that 
all this can be generalized to the curve ${\mathcal K}_m$ with defining 
polynomial $F=X^mY+Y^mZ+Z^mX$, provided that $\gcd(1,m^2-m+1)=1$.
In this case
$v_{P_0}(f_{ij})=-(m-1)i-mj$ and
$f_{ij}\in\cup_{m\geq 0}L(m P_0)$ if and only if 
$-i+(m-1)j\geq 0$. 
Since $(m-1)i+mj=(m-1)i'+mj'$ for some $(i',j')\neq(i,j)$ if and only if
$i\geq m$ or $j\geq m-1$ we deduce that 
$$\{-v_{P_0}(f_{ij}):f_{ij}\in \cup_{m\geq 0}L(m P_0)\}=\{(m-1)i+mj:(i,j)\neq 
(1,0),(2,0),\dots,(m-1,0)\}.$$
This set has exactly $\frac{m(m-1)}{2}$ gaps which is the genus of 
${\mathcal K}_m$. So it is exactly the Weierstrass semigroup at $P_0$. 
For further details on the Klein quartic we refer the reader to
\cite{Pretzel,HoLiPe:agc}.
\end{example}

\subsection{Basic notions and problems}

A \alert{numerical semigroup} is a subset $\Lambda$
of ${\mathbb N}_0$
containing $0$, closed
under summation and with finite complement in ${\mathbb N}_0$.
A general reference on numerical semigroups is \cite{RoGa:llibre}.

\subsubsection{Genus, conductor, gaps, non-gaps, enumeration}
For a numerical semigroup $\Lambda$
define
the \alert{genus} of $\Lambda$ as the number
$g=\#({\mathbb N}_0\setminus\Lambda)$
and the \alert{conductor} of $\Lambda$ as
the unique integer
$c\in\Lambda$
such that $c-1\not\in\Lambda$ and $c+{\mathbb N}_0\subseteq\Lambda$.
The elements in $\Lambda$ are called the \alert{non-gaps} of $\Lambda$
while the elements in
${\mathbb N}_0\setminus\Lambda$ are called the \alert{gaps} of $\Lambda$.
The \alert{enumeration}
of $\Lambda$ is the unique increasing bijective map
$\lambda:{\mathbb N}_0\longrightarrow\Lambda$.
We will use $\lambda_i$ for $\lambda(i)$.

\begin{lemma}
Let $\Lambda$ be a numerical semigroup with conductor $c$, genus $g$, and 
enumeration $\lambda$. The following are equivalent.
\begin{description}
\item[(i)] $\lambda_i\geq c$
\item[(ii)] $i\geq c-g$
\item[(iii)] $\lambda_i=g+i$
\end{description}
\end{lemma}

\begin{proof}
First of all notice that if
$g(i)$ is the number of gaps smaller than $\lambda_i$, then 
$\lambda_i=g(i)+i$.
To see that (i) and (iii) are equivalent notice that
$\lambda_i\geq c\Longleftrightarrow g(i)=g \Longleftrightarrow g(i)+i=g+i \Longleftrightarrow\lambda_i=g+i$.
Now, from this equivalence we deduce that $c=\lambda_{c-g}$.
Since $\lambda$ is increasing we deduce that $\lambda_i\geq c=\lambda_{c-g}$
if and only if $i\geq c-g$.
\end{proof}

%\subsubsection{Weierstrass semigroups}
%Let $F/{\mathbb F}_q$ be a function field and let $P$ be
%a rational point of $F/{\mathbb F}_q$.
%For a divisor $D$ of $F/{\mathbb F}_q$,
%let $$L(D)=\{0\}\cup\{f\in F^*: (f)+D\geq 0\}.$$
%The set $L(mP)$ is then the set of functions having only poles at $P%$
%and with pole order at most $m$.
%Define $A=\bigcup_{m\geq 0}L(m P)$, 
%that is, $A$ is 
%the ring of rational functions having poles only at $P$. Let $$\Lambda=\{-v_P(f): f\in A%\setminus\{0\}\}=\{-v_i: i\in{\mathbb N}_0\}$$ with
%$-v_i<-v_{i+1}$.
%It is well known that the number of elements in ${\mathbb N}_0$
%which are not in $\Lambda$ is equal to the genus of
%the function field.
%Moreover,
%$v_P(1)=0$ and $v_P(fg)=v_P(f)+v_P(g)$
%for all $f,g\in A$.
%Hence, $\Lambda$
%is a numerical semigroup. It is called the
%\alert{Weierstrass semigroup} at $P$.

\subsubsection{Generators, Ap\'ery set}
The \alert{generators} of a numerical semigroup are those non-gaps which can not
be obtained as a sum of two smaller non-gaps.
If $a_1,\dots,a_l$ 
are the generators of a semigroup $\Lambda$ then 
$\Lambda=\{n_1a_1+\dots+n_la_l:n_1,\dots,n_l\in{\mathbb N}_0\}$ and so 
$a_1,\dots,a_l$ are necessarily coprime.
If $a_1,\dots,a_l$ are coprime, we call $\{n_1a_1+\dots+n_la_l:n_1,\dots,n_l\in{\mathbb N}_0\}$
the \alert{semigroup generated} by $a_1,\dots,a_l$ and
denote it by $\langle a_1,\dots,a_l\rangle$.

The non-gap $\lambda_1$ is always a generator.
If for each integer $i$ from $0$ to $\lambda_1-1$ we consider $w_i$ to be the
smallest non-gap in $\Lambda$ that is congruent to $i$ modulo $\lambda_1$,
then each non-gap of $\Lambda$ can be expressed as $w_i+k\lambda_1$ for some
$i\in\{0,\dots,\lambda_1-1\}$ and some $k\in{\mathbb N}_0$. So, the generators different from $\lambda_1$ must be in 
$\{w_1,\dots,w_{\lambda_1-1}\}$ and this implies that there is always a finite number of generators. The set $\{w_0,w_1,\dots,w_{\lambda_1-1}\}$ is called the 
\alert{Ap\'ery set} of $\Lambda$ and denoted \alert{$Ap(\Lambda)$}.
It is easy to check that it equals
$\{l\in\Lambda: l-\lambda_1\not\in\Lambda\}.$
References related to the Ap\'ery set are \cite{Apery,FrGoHa,RoGaGaBr2002,RoGaGaBr2005,MaHe}.

\subsubsection{Frobenius' coin exchange problem}

Frobenius suggested the problem 
to determine the largest monetary amount that can not be obtained 
using only coins of specified denominations. 
A lot of information on the Frobenius' problem can be found in Ram\'\i rez Alfons\'\i n's book \cite{RamirezAlfonsin}.

If the different denominations are coprime then the set of amounts that can be obtained form a numerical semigroup and the question is equivalent to 
determining the largest gap.
This is why the largest gap of a numerical semigroup is called the \alert{Frobenius number} of the numerical semigroup.

If the number of denominations is two and the values of the coins are $a,b$ with $a,b$ coprime, then Sylvester's formula \cite{Sylvester} gives that the Frobenius number is 
$$ab-a-b.$$ 

However, when the number of denominations is larger, there is no closed polynomial form as can be derived from the next result due to Curtis~\cite{Curtis}.
\begin{theorem}
There is no finite set of polynomials $\{f_1,\dots,f_n\}$ such that
for each choice of $a,b,c\in{\mathbb N}$, there is some $i$ such that the Frobenius number of $a,b,c$ is $f_i(a,b,c)$.
\end{theorem}

\subsubsection{Hurwitz question}
It is usually attributed to Hurwitz the problem 
of determining whether there exist non-Weierstrass numerical semigroups,
to which Buchweitz gave a positive answer, and the problem of characterizing
Weierstrass semigroups. 
For these questions we refer the reader to \cite{Torres95,Kim96,Komeda98}
and all the citations therein.

A related problem %presented by Geil and Matsumoto 
is bounding the number of rational points of a curve 
using Weierstrass semigroups. 
One can find some bounds in \cite{StVo,Lewittes,GeilMatsumoto}
%In particular they define 
%$N_q (\Lambda) = \max\{N_q(F) | F is a function field over Fq having a rational place which Weierstrass semigroup equals Λ}.
%A bound can be found in \cite{GeilMatsumoto}.

\subsubsection{Wilf conjecture}
The Wilf conjecture (\cite{Wilf,DoMa})
states that the number $e$ of generators of a numerical semigroup of
genus $g$ and conductor $c$ satisfies
$$e\geq\frac{c}{c-g}.$$
It is easy to check it when the numerical semigroup is symmetric, that
is, when $c=2g$.
In \cite{DoMa} the inequality is proved for many
other cases.
In \cite{Bras:Fibonacci} it was proved by brute approach
that any numerical semigroup of genus at most $50$
also satisfies the conjecture.

\subsection{Classification}

\subsubsection{Symmetric and pseudo-symmetric numerical semigroups}
\label{section: symmetric numerical semigroups}

\begin{definition}A numerical semigroup $\Lambda$ with genus $g$ and conductor $c$ is said to be
\alert{symmetric} if $c=2g$.
%\index{semigroup!symmetric}
\end{definition}

Symmetric numerical semigroups have been widely studied. For instance in
\cite{KiPe,HoLiPe:agc,CaFa:semigroup_singular_plane_models,BAdM}.

\begin{example}
\label{exemple: semigroups gen by two int}
Semigroups
\alert{generated by two integers}
%\index{semigroup!generated by two integers}
are the semigroups of the form
$$\Lambda=\{ma+nb: a,b\in{\mathbb N}_0\}$$
for some integers $a$ and $b$.
For $\Lambda$ having 
finite complement in ${\mathbb N}_0$
it is necessary that $a$ and $b$ are
coprime integers.
Semigroups
generated by two coprime integers are 
symmetric \cite{KiPe,HoLiPe:agc}.

Geil introduces in
\cite{Geil:Norm-trace-codes}
the \alert{norm-trace curve} over ${\mathbb F}_{q^r}$
defined by the affine equation
$$x^{(q^r-1)/(q-1)}
=y^{q^{r-1}}+y^{q^{r-2}}
+\dots+y$$
where $q$ is a prime power.
It has a single rational point at infinity
and the Weierstrass semigroup at
the rational point at infinity
is generated by the two coprime integers
$(q^r-1)/(q-1)$ and
$q^{r-1}$.
So, it is an example
of a symmetric numerical semigroup.

Properties on semigroups
generated by two coprime integers can
be found in \cite{KiPe}.
For instance, the
semigroup generated by $a$ and $b$,
has conductor equal to $(a-1)(b-1)$,
and any element $l\in\Lambda$
can be written uniquely as
$l=ma+nb$ with $m,n$
integers such that $0\leq m<b$.

From the results in \cite[Section 3.2]{HoLiPe:agc}
one can get, for any
numerical semigroup $\Lambda$ generated by two integers,
the equation of a curve having a point whose
Weierstrass semigroup is $\Lambda$.
\end{example}

Let us state now a lemma related
to symmetric numerical semigroups.

\begin{lemma}
\label{lemma:symmetric-caracteritzacio}
\label{lemma:symmetric-implicacio}
A numerical semigroup $\Lambda$
with conductor $c$
is symmetric if and only if
for any non-negative integer $i$, if $i$ is a gap,
then $c-1-i$ is a non-gap.
\end{lemma}

The proof
can be found in
\cite[Remark~4.2]{KiPe}
and \cite[Proposition~5.7]{HoLiPe:agc}.
It follows by counting the number of gaps and non-gaps
smaller than the conductor and the fact that 
if $i$ is a non-gap then $c-1-i$ must be a 
gap because otherwise $c-1$ would also be a non-gap.

\begin{definition}
A numerical semigroup $\Lambda$ with genus $g$ and conductor $c$
is said to be \alert{pseudo-symmetric}
if $c=2g-1$.
\end{definition}

Notice that a symmetric numerical semigroup can not be pseudo-symmetric.
Next lemma as well as its proof is analogous to 
Lemma~\ref{lemma:symmetric-implicacio}. 

\begin{lemma}
\label{lemma:caracteritzaciopseudosimetrics}
A numerical semigroup $\Lambda$
with odd conductor $c$
is pseudo-symmetric if and only if
for any non-negative integer $i$ different from $(c-1)/2$, if $i$ is a gap,
then $c-1-i$ is a non-gap.
\end{lemma}

\begin{example}
%\label{example:klein}
%The Klein quartic
%\index{Klein quartic!curve}
%over ${\mathbb F}_q$
%is defined by the affine equation
%$$x^3y+y^3+x=0$$
%and it is non-singular
%if $\gcd(q,7)=1$.
%Suppose $\gcd(q,7)=1$ and
%denote $P_0$ the
%rational point with
%affine coordinates $x=0$ and $y=0$.
The Weierstrass semigroup
at $P_0$
of the Klein quartic 
of Example~\ref{example:klein}
is
$\Lambda=\{0,3\}\cup\{i\in{\mathbb N}_0: i\geq 5\}.$
In this case $c=5$
and the only gaps different from $(c-1)/2$ are
$l=1$ and $l=4$. In both cases we have
$c-1-l\in\Lambda$. This proves that
$\Lambda$ is pseudo-symmetric. 
\end{example}

In \cite{RoBr:irreducible} the authors prove that
the set of irreducible semigroups,
that is, the semigroups that can not be expressed as a proper intersection
of two numerical semigroups,
is the union of the set of symmetric semigroups and 
the set of pseudo-symmetric semigroups.

\subsubsection{Arf numerical semigroups}
\label{section: Arf numerical semigroups}

\begin{definition}A numerical semigroup $\Lambda$
with enumeration $\lambda$
is called an \alert{Arf numerical semigroup}
%\index{semigroup!Arf}
if
$\lambda_i+\lambda_j-\lambda_k\in\Lambda$
for every $i,j,k\in {\mathbb N}_0$ with $i\geq j\geq k$
\cite{CaFaMu:arf}.
\end{definition}

For further work on Arf numerical semigroups and generalizations we refer the reader to
\cite{BaDoFo,RoGaGaBr,BAGS,MuToVi:sparse}.
For results on Arf semigroups related to
coding theory, see \cite{Bras:2003_AAECC,CaFaMu:arf}.

\begin{example}
It is easy to check that the Weierstrass semigroup
in Example~\ref{example:klein} is Arf.
\end{example}

Let us state now two results
on Arf numerical semigroups
that will be used later.

\begin{lemma}
\label{lemma: arf te els non-gaps separats}
Suppose $\Lambda$ is Arf. If $i,i+j\in\Lambda$
for some $i,j\in{\mathbb N}_0$,
then
$i+kj\in\Lambda$ for all $k\in{\mathbb N}_0$.
Consequently,
if $\Lambda$ is Arf and $i,i+1\in\Lambda$, then $i\geq c$.
\end{lemma}

\begin{proof}
Let us prove this by induction on $k$.
It is obvious for $k=0$ and $k=1$.
If $k>0$ and $i,i+j,i+kj\in\Lambda$ then $(i+j)+(i+kj)-i=i+(k+1)j\in\Lambda$.
\end{proof}

Consequently, Arf semigroups are \alert{sparse semigroups} 
\cite{MuToVi:sparse}, that is, 
there are no two non-gaps in a row smaller than the conductor.

Let us give the definition
of inductive numerical semigroups.
They are an example of Arf numerical semigroups.

\begin{definition}A sequence $(H_n)$ of
numerical semigroups is called \alert{inductive}
if there exist sequences $(a_n)$ and $(b_n)$
of positive integers such that $H_1={\mathbb N}_0$ and for
$n>1$, $H_n=a_nH_{n-1}\cup\{m\in{\mathbb N}_0: m\geq a_nb_{n-1}\}$.
A numerical semigroup is called \alert{inductive} if it is a member of an
inductive sequence
\cite[Definition 2.13]{PeTo:redundancy_improved_codes}.
%\index{semigroup!inductive}
\end{definition}

One can see that inductive numerical semigroups are Arf
\cite{CaFaMu:arf}.

\begin{example}
\label{example:2nd GS es Arf}
Pellikaan, Stichtenoth and Torres proved in \cite{PeStTo}
that the
numerical
semigroups
for the
codes
over ${\mathbb F}_{q^2}$
associated to the second tower of Garcia-Stichtenoth
%\index{Garcia-Stichtenoth towers!semigroup}
attaining the Drinfeld-Vl\u adu\c t
bound
\cite{GaSt:tff}
are given
recursively by $\Lambda_1={\mathbb N}_0$
and, for $m>0$,
$$\Lambda_{m}=q\cdot\Lambda_{m-1}\cup\{i\in{\mathbb N}_0: i\geq
q^m-q^{\lfloor(m+1)/2\rfloor}\}.$$
They
are examples of inductive numerical semigroups
and hence, examples of Arf numerical semigroups.
\end{example}

\begin{example}
\alert{Hyperelliptic numerical semigroups}.
%\index{semigroup!hyperelliptic}
These are the
numerical semigroups
generated by $2$ and an odd integer.
They are of the form $$\Lambda=\{0,2,4,\dots,2k-2,2k,2k+1,2k+2,2k+3,\dots\}$$
for some positive integer $k$.
\end{example}

The next lemma is proved in \cite{CaFaMu:arf}.
\begin{lemma}
The only Arf symmetric semigroups are 
hyperelliptic semigroups.
\end{lemma}

%On the other hand
%there exist numerical semigroups which are neither Arf nor symmetric.
%For example, $$\Lambda=\{0,4,8,9,12,13,14,15,16,\dots\}.$$
%Its conductor is $12$ and its genus is $8$ and so it is not symmetric.
%It is not an Arf
%numerical semigroup
%by Lemma~\ref{lemma: arf te els non-gaps separats}.
In order to show which are the only Arf pseudo-symmetric semigroups
we need the Ap\'ery set that was previously defined.

\begin{lemma}
\label{lemma:pseudo}
Let $\Lambda$ be a pseudo-symmetric numerical semigroup. For any
$l\in Ap(\Lambda)$ 
different from $\lambda_1+(c-1)/2$,
$\lambda_1+c-1-l\in Ap(\Lambda)$.
\end{lemma}

\begin{proof}
Let us prove first that 
$\lambda_1+c-1-l\in \Lambda$.
Since 
$l\in Ap(\Lambda)$,
$l-\lambda_1\not\in\Lambda$ and it is different from 
$(c-1)/2$ by hypothesis. Thus
$\lambda_1+c-1-l=c-1-(l-\lambda_1)\in \Lambda$ 
because $\Lambda$ is pseudo-symmetric. 

Now, $\lambda_1+c-1-l-\lambda_1=c-1-l\not\in\Lambda$ 
because otherwise $c-1\in\Lambda$. So
$\lambda_1+c-1-l$ must belong to $Ap(\Lambda)$.
\end{proof}

\begin{lemma}
The only Arf pseudo-symmetric semigroups are
$\{0,3,4,5,6,\dots\}$ and
$\{0,3,5,6,7,\dots\}$ (corresponding to the Klein quartic).
\end{lemma}

\begin{proof}
Let $\Lambda$ be an Arf pseudo-symmetric
numerical semigroup.
Let us show first that 
$Ap(\Lambda)=\{0,\lambda_1+(c-1)/2,\lambda_1+c-1\}$.
The inclusion $\supseteq$ is obvious. In order to prove the opposite inclusion 
suppose $l\in Ap(\Lambda)$, $l\not\in\{0,\lambda_1+(c-1)/2,\lambda_1+c-1\}$.
By Lemma~\ref{lemma:pseudo}, $\lambda_1+c-1-l\in \Lambda$
and since $l\neq\lambda_1+c-1$, 
$\lambda_1+c-1-l\geq\lambda_1$.
On the other hand,
if $l\neq 0$ then $l\geq \lambda_1$.
Now, by the Arf condition, 
$\lambda_1+c-1-l+l-\lambda_1=c-1\in\Lambda$, which is a contradiction.

Now, if $\# Ap(\Lambda)=1$ then $\lambda_1=1$
and $\Lambda={\mathbb N}_0$. But ${\mathbb N}_0$ is not pseudo-symmetric.

If $\# Ap(\Lambda)=2$ then $\lambda_1=2$.
But then $\Lambda$ must be hyperelliptic and so $\Lambda$ is not 
pseudo-symmetric.

So $\# Ap(\Lambda)$ must be $3$.
This implies that $\lambda_1=3$ and that 
$1$ and $2$ are gaps.
If $1=(c-1)/2$ then $c=3$ and this gives $\Lambda=\{0,3,4,5,6,\dots\}$.
Else if $2=(c-1)/2$ then $c=5$ and this 
gives $\Lambda=\{0,3,5,6,7,\dots\}$.
Finally, if 
$1\neq(c-1)/2$ 
and 
$2\neq(c-1)/2$,
since $\Lambda$ is pseudo-symmetric,
$c-2,c-3\in\Lambda$. 
But this contradicts 
Lemma~\ref{lemma: arf te els non-gaps separats}.
\end{proof}

The next two lemmas are two characterizations of
Arf numerical semigroups. The first one is proved in
\cite[Proposition 1]{CaFaMu:arf}.

\begin{lemma}
\label{lemma: charactarf 2j-i}
The numerical semigroup $\Lambda$ with
enumeration $\lambda$ is Arf if and only if for every two positive integers $i,j$
with $i\geq j$,
$2\lambda_i-\lambda_j\in \Lambda$.
\end{lemma}

\begin{lemma}
The numerical semigroup $\Lambda$ is Arf if and only if
for any $l\in\Lambda$, the set
$S(l)=\{l'-l: l'\in\Lambda, l'\geq l\}$ is a
numerical
semigroup.
\end{lemma}

\begin{proof}
Suppose $\Lambda$ is Arf.
Then $0\in S(l)$ and
if $m_1=l'-l$, $m_2=l''-l$
with $l',l''\in \Lambda$
and $l'\geq l$,
$l''\geq l$,
then $m_1+m_2=l'+l''-l-l$.
Since $\Lambda$ is Arf,
$l'+l''-l\in\Lambda$
and it is larger than or equal to $l$.
Thus, $m_1+m_2\in S(l)$.
The finiteness of the complement 
of $S(l)$ is a consequence of the finiteness of the complement of $\Lambda$.

On the other hand,
if $\Lambda$ is such that
$S(l)$
is a numerical semigroup
for any $l\in\Lambda$
then, if $\lambda_i\geq\lambda_j\geq\lambda_k$
are in $\Lambda$, we will have
$\lambda_i-\lambda_k\in S(\lambda_k)$,
$\lambda_j-\lambda_k\in S(\lambda_k)$,
$\lambda_i+\lambda_j-\lambda_k-\lambda_k\in S(\lambda_k)$
and therefore $\lambda_i+\lambda_j-\lambda_k\in\Lambda$.
\end{proof}

\subsubsection{Numerical semigroups generated by an interval}
\label{section: interval numerical semigroups}

A numerical semigroup $\Lambda$
is \alert{generated by an interval}
%\index{semigroup!generated by an interval}
$\{i, i+1, \dots, j\}$
with $i,j\in{\mathbb N}_0$, $i\leq j$ if
$$\Lambda=\{n_i i+n_{i+1}(i+1)+\dots+n_j j :
n_i,n_{i+1},\dots,n_j\in{\mathbb N}_0\}.$$

A study of
semigroups
generated by intervals
was carried out
by Garc\'\i a-S\'anchez and
Rosales
in \cite{GaRo:interval}.

\begin{example}
The Weierstrass semigroup at the rational point at infinity
of the Hermitian curve (Example~\ref{example:hermite})
is generated by $q$ and $q+1$.
So, it is an example
of numerical semigroup generated by an interval.
\end{example}

\begin{lemma}
\label{lemma: forma semigrup generat per intervals}
The semigroup $\Lambda_{\{i,\dots,j\}}$
generated by the interval $\{i,i+1,\dots,j\}$
satisfies
$$\Lambda_{\{i,\dots,j\}}=\bigcup_{k\geq 0}\{ki,ki+1,ki+2,\dots ,kj\}.$$
\end{lemma}

This lemma is a reformulation of
\cite[Lemma 1]{GaRo:interval}.
In the same reference we can find the next result
\cite[Theorem 6]{GaRo:interval}.

\begin{lemma}
\label{lemma: intersection interval symmetric}
$\Lambda_{\{i,\dots,j\}}$
is symmetric
if and only if
$i\equiv 2\mbox{ mod } j-i.$
\end{lemma}

\begin{lemma}
The only numerical semigroups
which are generated by an interval and Arf,
are the 
semigroups which are equal to
$\{0\}\cup\{i\in{\mathbb N}_0: i\geq c\}$ for some non-negative integer $c$.
\end{lemma}

\begin{proof}It is a consequence of
Lemma~\ref{lemma: arf te els non-gaps separats}
and Lemma~\ref{lemma: forma semigrup generat per intervals}.
\end{proof}

\begin{lemma}
The unique numerical semigroup
which is pseudo-symmetric and generated by an interval 
is $\{0,3,4,5,6,\dots\}$.
\end{lemma}

\begin{proof}
By Lemma~\ref{lemma: forma semigrup generat per intervals},
for the non-trivial semigroup $\Lambda_{\{i,\dots,j\}}$
generated by the interval $\{i,\dots,j\}$, 
the intervals of gaps between $\lambda_0$
and the conductor satisfy that the length of each interval
is equal to the length of the previous interval minus
$j-i$.
On the other hand, the intervals of non-gaps
between $1$ and $c-1$ satisfy that the length of each interval is equal to the length of the previous interval plus $j-i$.

Now, by Lemma~\ref{lemma:caracteritzaciopseudosimetrics},
$(c-1)/2$ must be the first gap or the last gap of 
an interval of gaps.
Suppose that it is the first gap of an interval of $n$ gaps.
If it is equal to $1$ then $c=3$ and $\Lambda=\{0,3,4,5,6,\dots\}$.
Otherwise $(c-1)/2>\lambda_1$.
Then, if $\Lambda$ is pseudo-symmetric, the previous interval of non-gaps has length
$n-1$. Since $\Lambda$ is generated by an interval,
the first interval of non-gaps after 
$(c-1)/2$ must have length $n-1+j-i$ 
and since $\Lambda$ is pseudo-symmetric the interval of 
gaps before $(c-1)/2$ must have the same length. 
But since $\Lambda$ is generated by an interval,
the interval of gaps previous to $(c-1)/2$ must have length $n+j-i$.
This is a contradiction.
The same argument proves that 
$(c-1)/2$ can not be the last gap of an interval of gaps.
So, the only possibility for a pseudo-symmetric 
semigroup generated by an interval is when 
$(c-1)/2=1$, that is, when $\Lambda=\{0,3,4,5,6,\dots\}$.
\end{proof}

\subsubsection{Acute numerical semigroups}
\label{section: acute numerical semigroups}

\begin{definition}We say that a numerical semigroup is \alert{ordinary}
%\index{semigroup!ordinary}
if it is
equal to $$\{0\}\cup\{i\in{\mathbb N}_0: i\geq c\}$$ for some non-negative integer $c$.
\end{definition}

Almost all rational points on a curve of genus $g$ 
over an algebraically closed field have
Weierstrass semigroup
of the form
$\{0\}\cup\{i\in{\mathbb N}_0: i\geq g+1\}$.
Such points are said to be \alert{ordinary}.
This is why we call these numerical semigroups
ordinary \cite{Goldschmidt,FaKr,Stichtenoth:AFFaC}.
Caution must be taken when the characteristic of the ground field
is $p>0$, since there exist curves with infinitely many non-ordinary 
points \cite{StVo}.

Notice that ${\mathbb N}_0$ is an ordinary numerical semigroup.
It will be called the \alert{trivial} numerical semigroup.
%\index{semigroup!trivial}

\begin{definition}
Let $\Lambda$ be a numerical semigroup different from ${\mathbb N}_0$
with enumeration $\lambda$, genus $g$ and conductor $c$.
The element $\lambda_{\lambda^{-1}(c)-1}$ will be called
the \alert{dominant} 
%\index{dominant}
of the semigroup and will be denoted $d$.
For each $i\in{\mathbb N}_0$ let $g(i)$
be the number of gaps which are
smaller than $\lambda_i$. In particular,
$g(\lambda^{-1}(c))=g$ and
$g(\lambda^{-1}(d))=g'<g$.
If $i$ is the smallest integer for which
$g(i)=g'$ then $\lambda_i$
is called the \alert{subconductor}
%\index{subconductor}
of $\Lambda$ and denoted $c'$.
\end{definition}

\begin{remark}Notice that if $c'>0$, then $c'-1\not\in\Lambda$. Otherwise we would have
$g(\lambda^{-1}(c'-1))=g(\lambda^{-1}(c'))$ and $c'-1<c'$.
Notice also that all integers between $c'$ and $d$
are in $\Lambda$ because
otherwise $g(\lambda^{-1}(c'))<g'$.
\end{remark}

\begin{remark}
\label{remark: ordinary implica d,c' son 0}
For a numerical semigroup $\Lambda$ different from ${\mathbb N}_0$
the following are equivalent:
\begin{description}
\item[{\bf (i)}] $\Lambda$ is ordinary,
\item[{\bf (ii)}] the dominant of $\Lambda$ is $0$,
\item[{\bf (iii)}] the subconductor of $\Lambda$ is $0$.
\end{description}
Indeed,
{\bf (i)}$\Longleftrightarrow${\bf (ii)}
and {\bf (ii)}$\Longrightarrow${\bf (iii)}
are obvious.
Now, suppose {\bf (iii)} is satisfied.
If the dominant is larger than or equal to $1$
it means that $1$ is in $\Lambda$ and so $\Lambda={\mathbb N}_0$ a contradiction.
\end{remark}

\begin{definition}
If $\Lambda$ is a non-ordinary numerical semigroup
with enumeration $\lambda$ and
with subconductor $\lambda_i$ then
the element $\lambda_{i-1}$ will be called the \alert{subdominant} 
%\index{subdominant}
and
denoted $d'$.
\end{definition}

It is well defined because of Remark~\ref{remark: ordinary implica d,c' son 0}.

\begin{definition}\label{definition:acute}
A numerical semigroup $\Lambda$
is said to be \alert{acute}
%\index{semigroup!acute}
if
$\Lambda$ is ordinary or if $\Lambda$ is non-ordinary and
its conductor $c$, its subconductor $c'$, its dominant $d$ and
its subdominant $d'$
satisfy
$c-d\leq c'-d'$.
\end{definition}

Roughly speaking,
a numerical semigroup is acute if
the last interval of gaps before
the conductor is smaller than the previous interval of gaps.

\begin{example}
%\index{Hermitian!semigroup}
For the Hermitian curve over ${\mathbb F}_{16}$
the Weierstrass semigroup at the unique point
at infinity is
$$\{0,4,5,8,9,10\}\cup\{i\in{\mathbb N}_0: i\geq 12\}.$$
In this case
$c=12$, $d=10$, $c'=8$
and $d'=5$ and it
is easy to check that it is an acute numerical semigroup.
\end{example}

\begin{example}
%\index{Klein quartic!semigroup}
For the Weierstrass semigroup
at the rational point $P_0$
of the Klein quartic
in Example~\ref{example:klein}
we have
$c=5$, $d=c'=3$ and $d'=0$.
So, it is an example of a non-ordinary
acute numerical
semigroup.
\end{example}

\begin{lemma}
\label{lemma: exemples acute}
Let $\Lambda$ be a numerical semigroup.
\begin{enumerate}
\item If $\Lambda$ is symmetric then it is acute.
%\index{semigroup!symmetric}
\item If $\Lambda$ is pseudo-symmetric then it is acute.
\item If $\Lambda$ is Arf then it is acute.
%\index{semigroup!Arf}
\item If $\Lambda$ is generated by an interval then it is acute.
%\index{semigroup!generated by an interval}
\end{enumerate}
\end{lemma}

\begin{proof}
If $\Lambda$ is ordinary then it is obvious. Let us suppose that $\Lambda$
is a non-ordinary semigroup
with
genus $g$, conductor $c$, subconductor $c'$, dominant $d$ and subdominant~$d'$.
\begin{enumerate}
\item Suppose that $\Lambda$ is symmetric. We know
by Lemma~\ref{lemma:symmetric-caracteritzacio}
that a numerical semigroup $\Lambda$ is symmetric if and only if
for any non-negative integer $i$, if $i$ is a gap,
then $c-1-i\in\Lambda$. If moreover it is not ordinary, then $1$ is a gap. So,
$c-2\in\Lambda$ and it is precisely the dominant.
Hence, $c-d=2$. Since $c'-1$ is a gap,
$c'-d'\geq 2=c-d$
and so $\Lambda$ is acute.
\item 
Suppose that $\Lambda$ is pseudo-symmetric.
If $1=(c-1)/2$ 
then $c=3$ and $\Lambda=\{0,3,4,5,6,\dots\}$ which is ordinary.
Else if $1\neq(c-1)/2$ then the proof 
is equivalent to the one for symmetric semigroups.
\item Suppose $\Lambda$ is Arf.
Since $d\geq c'> d'$, then
$d+c'-d'$ is in $\Lambda$
and it is strictly larger than the dominant $d$. Hence it is larger than
or equal
to $c$. So,
$d+c'-d'\geq c$
and $\Lambda$ is acute.
\item Suppose that $\Lambda$ is generated by the interval
$\{i, i+1, \dots , j\}$. Then, by
Lemma~\ref{lemma: forma semigrup generat per intervals},
there exists $k$ such that
$c=k i$, $c'=(k-1) i$, $d=(k-1) j$ and
$d'=(k-2) j$.
So, $c-d=k(i-j)+j$ while
$c'-d'=k(i-j)-i+2j$. Hence,
$\Lambda$ is acute.
\end{enumerate}
\end{proof}

\begin{figure}
\centering
\setlength{\unitlength}{.5cm}
\newcommand\ambcolor[1]{#1}
\newcommand\black[1]{#1}

{\small
\begin{picture}(16.,20.)
%\thicklines

\put(8.,19.){\makebox(0,0){\ambcolor{\fbox{\black{Numerical semigroups}}}}}

\ambcolor{
\qbezier(8.,18.5)(8.,18.5)(8.,16.45)}

\put(8.,16.){\makebox(0,0){\ambcolor{\fbox{\black{Acute}}}}}
\ambcolor{\qbezier(8.,15.5)(8.,15.5)(8.,13.45)}

\put(8.,13.){\makebox(0,0){\ambcolor{\fbox{\black{Irreducible}}}}}
\ambcolor{
\qbezier(8.,12.5)(8.,12.5)(6.,10.5)
\qbezier(8.,12.5)(8.,12.5)(10.,10.5)}

\ambcolor{
\qbezier(8.,15.5)(8.,15.6)(2.,10.45)
\qbezier(8.,15.5)(8.,15.6)(14.,10.55)
}

\put(6.,10.){\makebox(0,0){\ambcolor{\fbox{\black{Symmetric}}}}}
\put(6.,8.8){\makebox(0,0){{\tiny\shortstack[b]{EX: Hermitian c.\\Norm-Trace c.}}}}
\put(10.,10.){\makebox(0,0){\ambcolor{\fbox{\black{Pseudo-sym}}}}}
\put(10.,8.8){\makebox(0,0){{\tiny\shortstack[b]{EX: Klein quartic.}}}}
\put(2.,10.){\makebox(0,0){\ambcolor{\fbox{\black{Arf}}}}}
\put(2.,8.8){\makebox(0,0){{\tiny\shortstack[b]{EX: Klein quartic.\\Gar-Sti  tower}}}}
\put(14.,10.){\makebox(0,0){\ambcolor{\fbox{\black{$\Lambda_{\{i,\dots,j\}}$}}}}}
\put(14.,8.8){\makebox(0,0){{\tiny\shortstack[b]{EX: Hermitian c.}}}}

\ambcolor{
\qbezier(6.,8.)(6.,8.)(4.5,6.)
\qbezier(2.,8.)(2.,8.)(8.,6.)
\qbezier(14.,8.)(14.,8.)(11.5,6.2)
\qbezier(6.,8.)(6.,8.)(11.5,6.2)
\qbezier(2.,8.)(2.,8.)(4.5,6.)
\qbezier(14.,8.)(14.,8.)(8.,6.)
%pseudosimetrics
\qbezier(10.,8.3)(2.,7.6)(0.7,4.9)
\qbezier(2.,8.)(2.,8.)(0.7,4.9)
\qbezier(10.,8.3)(14.,7.4)(15.5,4.5)
\qbezier(14.,8.)(14.,8.)(15.5,4.5)
}

\put(1.,4.){\makebox(0,0){\ambcolor{\fbox{\black{\shortstack[b]{Klein q.\\$\{0,3,4,\dots\}$}}}}}}
\put(4.2,5.5){\makebox(0,0){\ambcolor{\fbox{\black{Hyperelliptic}}}}}
\put(4.6,4.3){\makebox(0,0){{\tiny\shortstack[b]{Campillo, Farran,\\Munuera}}}}
\put(7.9,5.5){\makebox(0,0){\ambcolor{\fbox{\black{Ordinary}}}}}
\put(11.8,5.5){\makebox(0,0){\ambcolor{\fbox{\black{$\Lambda_{\{i,\dots,\frac{(k+1)i-2}{k}\}}$}}}}}
\put(11.6,4.3){\makebox(0,0){{\tiny\shortstack[b]{Garc\'\i a-S\'anchez,\\Rosales}}}}
\put(15.,4.){\makebox(0,0){\ambcolor{\fbox{\black{$\{0,3,4,\dots\}$}}}}}

\ambcolor{
\qbezier(4.5,3.8)(4.5,3.8)(8.,0.5)
\qbezier(8.,5.)(8.,5.)(8.,0.5)
\qbezier(11.5,3.8)(11.5,3.8)(8.,0.5)
}

\put(8.,0.){\makebox(0,0){\ambcolor{\fbox{\black{Trivial: ${\mathbb N}_0$}}}}}
\end{picture}
}

\caption{Diagram of semigroup classes and inclusions.}
\label{figure:inclosions}
\end{figure}
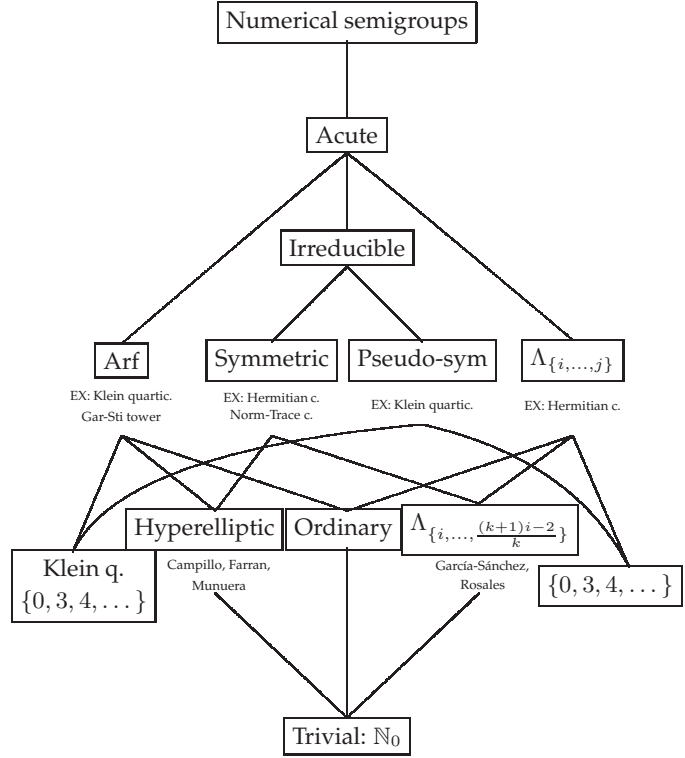

In Figure~\ref{figure:inclosions} we summarize
all the relations
we have proved
between
acute semigroups,
symmetric and pseudo-symmetric semigroups,
Arf semigroups
and semigroups generated by an interval.

\begin{remark}There exist numerical semigroups which are not acute.
For instance, $$\Lambda=\{0,6,8,9\}\cup\{i\in{\mathbb N}_0: i\geq 12\}.$$
In this case,
$c=12$, $d=9$, $c'=8$ and $d'=6$.

On the other hand there exist
numerical semigroups which are
acute and which are not
symmetric, pseudo-symmetric, Arf or interval-generated.
For example,
$$\Lambda=\{0,10,11\}\cup\{i\in{\mathbb N}_0: i\geq 15\}.$$
In this case,
$c=15$, $d=11$, $c'=10$ and $d'=0$.
\end{remark}

\subsection{Characterization}

\subsubsection{Homomorphisms of semigroups}

\alert{Homomorphisms of numerical semigroups}, that is, maps $f$ between numerical semigroups such that $f(a+b)=f(a)+f(b)$,
are exactly the scale maps $f(a)=ka$ for all $a$, for some constant $k\geq 0$.
Indeed, if $f$ is a homomorphism then
 $\frac{f(a)}{a}$ is constant since $f(ab)=a\cdot f(b)=b\cdot f(a)$.
Furthermore, the unique surjective homomorphism is the identity.
Indeed, for a semigroup $\Lambda$, the set $k\Lambda$ is a numerical semigroup
only if $k=1$.

\subsubsection{The $\oplus$ operation, the $\nu$ sequence, and the $\tau$ sequence}

Next we define three important objects describing the addition behavior of a numerical semigroup.

\begin{definition}
The operation $\oplus_\Lambda:{\mathbb N}_0\times{\mathbb N}_0\rightarrow{\mathbb N}_0$
associated to the numerical semigroup $\Lambda$ (with enumeration $\lambda$)
is defined as
$$i\oplus_\Lambda j=\lambda^{-1}(\lambda_i+\lambda_j)$$
for any $i,j\in{\mathbb N}_0$.
Equivalently,
$$\lambda_i+\lambda_j=\lambda_{i\oplus_\Lambda j}.$$
\end{definition}

The subindex referring to the semigroup
may be ommitted if the semigroup is clear by the context.
The operation $\oplus$ is obviously commutative, associative,
and has $0$ as identity element. However there is in general no inverse element.
Also, the operation $\oplus$ is compatible with the natural order of
${\mathbb N}_0$. That is, if $a<b$ then $a\oplus c<b\oplus c$ and $c\oplus
a<c\oplus b$ for any $c\in{\mathbb N}_0$.

\begin{example}
\label{example:oplus}
For the numerical semigroup
$\Lambda=\{ 0, 4, 5, 8, 9, 10, 12, 13, 14, \dots \}$
the first values of $\oplus$ are given in the next table:

$$\begin{array}{c|ccccccccc}
\oplus &0 &1 &2 &3 &4 &5 &6 &7 &\dots\\\hline
0&  0 &  1 &  2 &  3 &  4 &  5 &  6 &  7 &\dots\\
1&  1 &  3 &  4 &  6 &  7 &  8 & 10 & 11 &\dots\\
2&  2 &  4 &  5 &  7 &  8 &  9 & 11 & 12 &\dots\\
3&  3 &  6 &  7 & 10 & 11 & 12 & 14 & 15 &\dots\\
4&  4 &  7 &  8 & 11 & 12 & 13 & 15 & 16 &\dots\\
5&  5 &  8 &  9 & 12 & 13 & 14 & 16 & 17 &\dots\\
6&  6 & 10 & 11 & 14 & 15 & 16 & 18 & 19 &\dots\\
7&  7 & 11 & 12 & 15 & 16 & 17 & 19 & 20 &\dots\\
\vdots&\vdots&\vdots&\vdots&\vdots&\vdots&\vdots&\vdots&\vdots&\ddots\\
\end{array}$$

\end{example}

\begin{definition}
The partial ordering $\pleq_\Lambda$ on ${\mathbb N}_0$
associated to the numerical semigroup $\Lambda$
is defined as follows: 
$$i\pleq_\Lambda j \mbox{ if and only if } \lambda_j-\lambda_i\in\Lambda.$$
Equivalently
$$i\pleq_\Lambda j \mbox{ if and only if there exists } k \mbox{ with } i\oplus_\Lambda k=j.$$
\end{definition}

As before, if $\Lambda$ is clear by the context then the subindex may be ommitted.

\begin{definition}
Given a numerical semigroup $\Lambda$,
the set $N_i$ is defined by 
$$N_i=\{j\in{\mathbb N}_0:j\pleq_\Lambda i\}=\{j\in{\mathbb N}_0:\lambda_i-\lambda_j\in\Lambda\}.$$ 
The sequence $\nu_i$ is defined by $\nu_i=\#N_i$.
\end{definition}

\begin{example}
\label{example:nutrivial}
The $\nu$ sequence of the trivial semigroup $\Lambda={\mathbb N}_0$ is
$1,2,3,4,\dots$.
\end{example}

Next lemma states the relationship between $\oplus$ and the
$\nu$ sequence of a numerical semigroup. Its proof is obvious.

\begin{lemma}
\label{lemma:nu_from_oplus}
Let $\Lambda$ be a numerical semigroup and $\nu$ the corresponding
$\nu$ sequence. For all $i\in{\mathbb N}_0$,
$$\nu_i=\#\{(j,k)\in{\mathbb N}_0^2:j\oplus k=i\}.$$
\end{lemma}

\begin{example}
If we are given the table in Example~\ref{example:oplus}
we can deduce $\nu_7$ by counting the number of
occurrences of $7$ inside the table. This is exactly $6$.
Thus, $\nu_7=6$.
Repeating the same process for all $\nu_i$ with $i< 7$
we get $\nu_0=1$, $\nu_1=2$, $\nu_2=2$, $\nu_3=3$, $\nu_4=4$,
$\nu_5=3$, $\nu_6=4$.
\end{example}

This lemma implies that any finite set in the sequence $\nu$
can be determined by a finite set of $\oplus$ values.
Indeed, to compute $\nu_i$ it is enough to know $\{j\oplus k :0\leq j,k \leq i\}$.

\begin{definition}
Given a numerical semigroup $\Lambda$ 
define its \alert{$\tau$ sequence} by
$$\tau_i=\max\{j\in{\mathbb N}_0:\mbox{ exists }k\mbox{ with }j\leq k\leq i\mbox{ and }j\oplus_\Lambda k=i\}.$$
\end{definition}
Notice that 
%if $N_i=\{j\in{\mathbb N}_0:\lambda_i-\lambda_j\in\Lambda\}$ then 
$\tau_i$ is the largest element $j$ in $N_i$ with $\lambda_j\leq \lambda_i/2$.
In particular, if $\lambda_i/2\in\Lambda$ then $\tau_i=\lambda^{-1}(\lambda_i/2)$.
Notice also that $\tau_i$ is $0$ if and only if $\lambda_i$ is either $0$
or a generator of $\Lambda$.

\begin{example}
In Table~\ref{table}
we show the $\nu$ sequence and the $\tau$ sequence
of the numerical semigroups generated by $4,5$
and generated by $6,7,8,17$.

\begin{table}
\centering
\subtable[$<4,5>$]
{
$\begin{array}{|cccc|}
\hline
i & \lambda_i & \nu_i & \tau_i\\
\hline
0&0&1&0\\
1&4&2&0\\
2&5&2&0\\
3&8&3&1\\
4&9&4&1\\
5&10&3&2\\
6&12&4&1\\
7&13&6&2\\
8&14&6&2\\
9&15&4&2\\
10&16&5&3\\
11&17&8&3\\
12&18&9&4\\
13&19&8&4\\
14&20&9&5\\
15&21&10&4\\
16&22&12&5\\
17&23&12&5\\
18&24&13&6\\
19&25&14&6\\
20&26&15&7\\
21&27&16&7\\
22&28&17&8\\
23&29&18&8\\
\vdots&\vdots&\vdots&\vdots
\\\hline
\end{array}$}
%\caption{}
%\end{table}
%\ \ \ \ \
%\begin{table}
%\label{table:67817}
\subtable[$<6,7,8,17>$]%$\nu$ and $\tau$ sequences of the numerical semigroup generated by $6,7,8,17$]
{$\begin{array}{|cccc|}
\hline
i & \lambda_i & \nu_i & \tau_i\\
\hline
0&0&1&0\\
1&6&2&0\\
2&7&2&0\\
3&8&2&0\\
4&12&3&1\\
5&13&4&1\\
6&14&5&2\\
7&15&4&2\\
8&16&3&3\\
9&17&2&0\\
10&18&4&1\\
11&19&6&2\\
12&20&8&3\\
13&21&8&3\\
14&22&8&3\\
15&23&8&3\\
16&24&9&4\\
17&25&10&4\\
18&26&11&5\\
19&27&12&5\\
20&28&13&6\\
21&29&14&6\\
22&30&15&7\\
23&31&16&7\\
\vdots&\vdots&\vdots&\vdots\\
\hline
\end{array}$}
\caption{$\nu$ and $\tau$ sequences of the numerical semigroups generated by $4,5$ and
by $6,7,8,17$}
\label{table}
\end{table}

%\noindent
%\resizebox{\textwidth}{!}{$\begin{array}{|c|ccccccccccccccccccccccccc|}
%\hline
%i&0&1&2&3&4&5&6&7&8&9&10&11&12&13&14&15&16&17&18&19&20&21&22&23&\dots\\
%\lambda_i&0&4&5&8&9&10&12&13&14&15&16&17&18&19&20&21&22&23&24&25&26&27&28&29&\dots\\
%\nu_i&1&2&2&3&4&3&4&6&6&4&5&8&9&8&9&10&12&12&13&14&15&16&17&18&\dots\\
%\tau_i&0&0&0&1&1&2&1&2&2&2&3&3&4&4&5&4&5&5&6&6&7&7&8&8&\dots\\
%\hline
%\end{array}$}

%\medskip

%In next table we show the $\nu$ sequence and the $\tau$ sequence
%of the numerical semigroup generated by $6,7,8,17$.

%\medskip

%\noindent
%\resizebox{\textwidth}{!}{$\begin{array}{|c|ccccccccccccccccccccccc|}
%\hline
%i&0&1&2&3&4&5&6&7&8&9&10&11&12&13&14&15&16&17&18&19&20&21&\dots\\
%\lambda_i&0&6&7&8&12&13&14&15&16&17&18&19&20&21&22&23&24&25&26&27&28&29&\dots\\
%\nu_i&1&2&2&2&3&4&5&4&3&2&4&6&8&8&8&8&9&10&11&12&13&14&\dots\\
%\tau_i&0&0&0&0&1&1&2&2&3&0&1&2&3&3&3&3&4&4&5&5&6&6&\dots\\
%\hline
%\end{array}$}

\end{example}

One difference between the $\tau$ sequence and the $\nu$ sequence is that,
while in the $\nu$ sequence not all non-negative integers need to appear,
in the $\tau$ sequence all of them appear.
Notice for instance that $7$ does not appear in the $\nu$ sequence of the
numerical semigroup generated by $4$ and $5$ nor the numerical semigroup
generated by $6,7,8,17$.
The reason for which any non negative integer $j$ appears
in the $\tau$ sequence is that if $\lambda_i=2\lambda_j$ then $\tau_i=j$.
Furthermore,
the smallest $i$ for which $\tau_i=j$ corresponds to $\lambda_i=2\lambda_j$.

\subsubsection{Characterization of a numerical semigroup by $\oplus$}

The next result was proved in \cite{Bras:2005_AGCT,Bras:2007_IEEE_ANote}.

\begin{lemma}\label{lemma:characterization oplus}
The $\oplus$ operation uniquely determines a semigroup.
\end{lemma}

\begin{proof}
Suppose that two semigroups $\Lambda=\{\lambda_0<\lambda_1<\dots\}$ and
$\Lambda'=\{\lambda'_0<\lambda'_1<\dots\}$ have the same associated operation $\oplus$.
Define the map $f(\lambda_i)=\lambda'_i$. It is obviously surjective and it is a homomorphism since
$f(\lambda_i+\lambda_j)=f(\lambda_{i\oplus j})=\lambda'_{i\oplus j}=\lambda'_i+\lambda'_j=f(\lambda_i)+f(\lambda_j)$. So, $\Lambda=\Lambda'$.
\end{proof}

Conversely to Lemma~\ref{lemma:characterization oplus}
we next prove that any finite set of $\oplus$ values
is shared by an infinite number of semigroups.
This was proved in \cite{Bras:2007_IEEE_ANote}.

\begin{lemma}
\label{lemma:oplus2}
Let $a,b$ be positive integers.
Let $\Lambda$ be a numerical semigroup with enumeration $\lambda$
and let $d$ be an integer with $d\geq 2$. Define the numerical semigroup
$\Lambda'=d\Lambda\cup\{i\in{\mathbb N}: i\geq
d\lambda_{a\oplus b}\}$.
%If we denote $\oplus_\Lambda$ and $\oplus_{\Lambda'}$
%the $\oplus$-operations corresponding to $\Lambda$ and $\Lambda'$ respectively,
%then
Then
$i\oplus_{\Lambda'}
j=i\oplus_\Lambda j$
for all $i\leq a$ and all $j\leq b$, and $\Lambda'\neq\Lambda$.
\end{lemma}

\begin{proof}
It is obious that $\Lambda'\neq\Lambda$.
Let $\lambda'$ be the enumeration of $\Lambda'$.
For all $k\leq a\oplus_{\Lambda} b$, $\lambda'_k=d\lambda_k$.
In particular, if $i\leq a$ and $j\leq b$ then
$\lambda'_i=d\lambda_i$ and $\lambda'_j=d\lambda_j$.
Hence,
$\lambda'_{i\oplus_{\Lambda'} j}=\lambda'_i+\lambda'_j=d\lambda_i+d\lambda_j=d\lambda_{i\oplus_\Lambda
  j}=\lambda'_{i\oplus_\Lambda j}$. This implies
$i\oplus_{\Lambda'} j=i\oplus_\Lambda j$.
\end{proof}

By varying $d$ in Lemma~\ref{lemma:oplus2},
we can see that
although the values $(i\oplus j)_{0\leq i,j}$
of a numerical semigroup uniquely determine it,
any subset $(i\oplus j)_{0\leq i\leq a, 0\leq j\leq b}$
is exactly the corresponding subset of infinitely many
numerical semigroups.

\subsubsection{Characterization of a numerical semigroup by $\nu$}
\label{section: nu determina lambda}

We will use the following result on the values $\nu_i$.
It can be found in \cite[Theorem 3.8.]{KiPe}.

\begin{lemma}
\label{lemma:nu}
Let $\Lambda$ be a numerical semigroup with genus $g$, conductor $c$ and enumeration $\lambda$.
Let $g(i)$
be the number of gaps
smaller than $\lambda_i$
and let
$$D(i)=\{l\in{\mathbb N}_0\setminus\Lambda: \lambda_i-l\in{\mathbb N}_0\setminus\Lambda\}.$$
Then, for all $i\in{\mathbb N}_0$,
$$\nu_i=i-g(i)+\#D(i)+1.$$
In particular,
for all $i\geq 2c-g-1$
(or equivalently, for all $i$ such that
$\lambda_i\geq 2c-1$),
$\nu_i=i-g+1.$
\end{lemma}

\begin{proof}
The number of gaps smaller than $\lambda_i$ is $g(i)$ but it is also 
$\#D(i)+\#\{l\in{\mathbb N}_0\setminus\Lambda: \lambda_i-l\in\Lambda\}$,
so,
\begin{equation}
\label{oneside}
g(i)=\#D(i)+\#\{l\in{\mathbb N}_0\setminus\Lambda: \lambda_i-l\in\Lambda\}.
\end{equation}
On the other hand,
the number of non-gaps which are at most $\lambda_i$ is $i+1$ but it is also
$\nu_i+\#\{l\in\Lambda: \lambda_i-l\in{\mathbb N}_0\setminus\Lambda\}=\nu_i+\#\{l\in{\mathbb N}_0\setminus\Lambda: \lambda_i-l\in\Lambda\}$. So,
\begin{equation}
\label{otherside}
i+1=\nu_i+\#\{l\in{\mathbb N}_0\setminus\Lambda: \lambda_i-l\in\Lambda\}.
\end{equation}
From equalities
\ref{oneside} and \ref{otherside}
we get 
$$g(i)-\#D(i)=i+1-\nu_i$$
which leads to the desired result.
\end{proof}

The next lemma shows that if a numerical semigroup is non-trivial 
then there exists at least one value $k$ such that $\nu_k=\nu_{k+1}$.

\begin{lemma}
\label{lemma:repeatednu}
Suppose $\Lambda\neq{\mathbb N}_0$
and suppose that $c$ and $g$ are the conductor and the genus of $\Lambda$.
Let $k=2c-g-2$. Then $\nu_{k}=\nu_{k+1}$.
\end{lemma}

\begin{proof}
Since $\Lambda\neq{\mathbb N}_0$,
$c\geq 2$ and so
$2c-2\geq c$.
This implies $k=\lambda^{-1}(2c-2)$ and $g(k)=g$.
By Lemma~\ref{lemma:nu},
$\nu_{k}=k-g+\#D(k)+1$.
But $D(k)=\{c-1\}$.
So, $\nu_k=k-g+2$.
On the other hand, 
$g(k+1)=g$ and
$D(k+1)=\emptyset$.
By Lemma~\ref{lemma:nu} again,
$\nu_{k+1}=k-g+2=\nu_k$. 
\end{proof}

\begin{lemma}
\label{lemma:nutrivialunic}
The trivial semigroup is the unique numerical semigroup with $\nu$ sequence
equal to $1,2,3,4,5,\dots$.
\end{lemma}

\begin{proof}
As a consequence of Lemma~\ref{lemma:repeatednu} 
for any other numerical semigroup there is
a value in the $\nu$ sequence that appears at least three times.
\end{proof}

The next result was proved in \cite{Bras:2004_IEEE_Acute,Bras:2005_AGCT}.

%In this section we prove that any numerical semigroup
%is uniquely determined by the associated sequence $(\nu_i)$.
%We will use Lemma~\ref{lemma:nu} 
%to prove the one-to-one correspondence
%between semigroups and sequences $(\nu_i)$.

\begin{theorem}\label{theorem: nu determina lambda}
The $\nu$ sequence of a numerical semigroup determines it.
\end{theorem}

\begin{proof}
If $\Lambda={\mathbb N}_0$ then, by 
Lemma~\ref{lemma:nutrivialunic},
its $\nu$ sequence is unique.

Suppose that $\Lambda\neq{\mathbb N}_0$.
Then we can determine the genus and the conductor
from the $\nu$ sequence.
Indeed, let $k=2c-g-2$. 
In the following we will show how to determine
$k$ without the knowledge of $c$ and $g$.
By Lemma~\ref{lemma:repeatednu}
it holds that $\nu_k=\nu_{k+1}=k-g+2$ and by
Lemma~\ref{lemma:nu} 
$\nu_i=i-g+1$ for all $i>k$, which means
that $\nu_{i+1}=\nu_{i}+1$ for all $i>k$.
So,
$$k=\max\{i: \nu_i=\nu_{i+1}\}.$$
We can now determine the genus as
$$g=k+2-\nu_{k}$$
and the conductor as
$$c=\frac{k+g+2}{2}.$$
At this point we know that $\{0\}\in\Lambda$
and $\{i\in{\mathbb N}_0: i\geq c\}\subseteq\Lambda$
and, furthermore,
$\{1,c-1\}\subseteq{\mathbb N}_0\setminus\Lambda$.
It remains to determine 
for all $i\in\{2,\dots,c-2\}$
whether
$i\in\Lambda$. 
Let us assume
$i\in\{2,\dots,c-2\}$.

On one hand, $c-1+i-g>c-g$ and so
$\lambda_{c-1+i-g}>c$.
This means that
$g(c-1+i-g)=g$ and hence
\begin{equation}
\label{equacio: demo nu determina lambda 1}
\nu_{c-1+i-g}=c-1+i-g-g+\#D(c-1+i-g)+1.
\end{equation}

On the other hand, if we define
$\tilde{D}(i)$ to be
$$\tilde{D}(i)=\{l\in{\mathbb N}_0\setminus\Lambda: c-1+i-l\in{\mathbb N}_0\setminus\Lambda,i<l<c-1\}$$
then
\begin{equation}
\label{equacio: demo nu determina lambda 2}
D(c-1+i-g)=\casos{\tilde{D}(i)\cup\{c-1,i\}}{if $i\in{\mathbb N}_0\setminus\Lambda$}{\tilde{D}(i)}{otherwise.}
\end{equation}
So, from (\ref{equacio: demo nu determina lambda 1}) and
(\ref{equacio: demo nu determina lambda 2}),
\begin{center}
$i$ is a non-gap $\Longleftrightarrow
\nu_{c-1+i-g}=c+i-2g+\#\tilde{D}(i).$
\end{center}
This gives an inductive procedure to decide
whether $i$ belongs to $\Lambda$
decreasingly from $i=c-2$ to $i=2$.
\end{proof}

\begin{remark}
From the proof of
Theorem~\ref{theorem: nu determina lambda}
we see that a semigroup can be determined by
$k=\max\{i: \nu_i=\nu_{i+1}\}$ and
the values $\nu_i$ for $i\in\{c-g+1,\dots,2c-g-3\}$.
\end{remark}

\begin{remark}
Lemma~\ref{lemma:characterization oplus} is a consequence of Lemma~\ref{lemma:nu_from_oplus} and
Theorem~\ref{theorem: nu determina lambda}.
\end{remark}

Conversely to Theorem~\ref{theorem: nu determina lambda}
we next prove that any finite set of $\nu$ values
is shared by an infinite number of semigroups.
This was proved in  \cite{Bras:2007_IEEE_ANote}.
Thus, the construction just given 
to determine a numerical semigroup from
its $\nu$ sequence can only be performed
if we know the behavior of the infinitely many values in the
$\nu$ sequence.

\begin{lemma}
\label{lemma:nu2}
Let $k$ be a positive integer.
Let $\Lambda$ be a numerical semigroup with enumeration $\lambda$
and let $d$ be an integer with $d\geq 2$.
Define the semigroup $\Lambda'=d\Lambda\cup\{i\in{\mathbb N}: i\geq
d\lambda_k\}$
and let $\nu^{\Lambda}$ and $\nu^{\Lambda'}$ be the $\nu$ sequence corresponding to
$\Lambda$ and $\Lambda'$ respectively.
Then $\nu_i^{\Lambda'}=\nu_i^{\Lambda}$
for all $i\leq k$ and
$\Lambda'\neq \Lambda$.
\end{lemma}

\begin{proof}
It is obvious that $\Lambda'\neq \Lambda$.
Let $\lambda'$ be the enumeration of $\Lambda'$.
For all $i\leq k$, $\lambda'_i=d\lambda_i$.
In particular, if $j\leq i\leq k$, then
$\lambda'_i-\lambda'_j=d(\lambda_i-\lambda_j)\in\Lambda'\Longleftrightarrow
\lambda_i-\lambda_j\in\Lambda$.
Hence, $\nu_i^{\Lambda'}=\nu_i^{\Lambda}$.
\end{proof}

As a consequence of Lemma~\ref{lemma:nu2},
although the sequence $\nu$ of a numerical semigroup
uniquely determines it, any subset $(\nu_i)_{0\leq i\leq k-1}$ is
exactly the set of the first $k$ values of the $\nu$ sequence of
infinitely many semigroups. In fact, by varying $d$ among
the positive integers, we get an infinite set of semigroups, all of
them sharing the first $k$ values in the $\nu$ sequence.

It would be interesting to find which sequences of
positive integers correspond to the sequence $\nu$ of a
numerical semigroup. By now, only
some necessary conditions can be stated, for instance,
\begin{itemize}
\item
$\nu_0=1$,
\item
$\nu_1=2$,
\item
$\nu_i\leq  i+1$ for all $i\in{\mathbb N}_0$,
\item
there exists $k$ such that $\nu_{i+1}=\nu_i+1$ for all $i\geq k$,
%\item
%by \cite[Theorem 7.3]{Bras:acute},
%if $\nu$ is non-decreasing then it must be
%$$1,2,2,\dots,2,3,4,5,6,7,8,9,10,\dots$$
\end{itemize}

\subsubsection{Characterization of a numerical semigroup
by $\tau$}
\label{section:characterization}

In this section we show that a numerical semigroup
is determined by its $\tau$ sequence.
This was proved in \cite{Bras:2009_DCC}.

\begin{lemma}
\label{lemma:precharacterization}
Let $\Lambda$ be a numerical semigroup with
enumeration $\lambda$,
conductor $c>2$, genus $g$, and dominant $d$.
Then
\begin{enumerate}
\item $\tau_{(2c-g-2)+2i}=\tau_{(2c-g-2)+2i+1}=(c-g-1)+i$ for all $i\geq 0$
\item At least one of the following statements holds
\begin{itemize}
\item
$\tau_{(2c-g-2)-1}=c-g-1$
\item
$\tau_{(2c-g-2)-2}=c-g-1$
\end{itemize}
\end{enumerate}
\end{lemma}

\begin{proof}
\begin{enumerate}
\item
If $i\geq 1$ then $\lambda_{(2c-g-2)+2i}=2c-2+2i$ and
$\lambda_{(2c-g-2)+2i}/2=c-1+i\in\Lambda$. So
$\tau_{(2c-g-2)+2i}=\lambda^{-1}(c-1+i)=c-1+i-g$.
On the other hand, $\lambda_{(2c-g-2)+2i+1}=2c-2+2i+1=(c-1+i)+(c-1+i+1)$
and so $\tau_{(2c-g-2)+2i+1}=\lambda^{-1}(c-1+i)=c-1+i-g$.

If $i=0$ then $\lambda_{(2c-g-2)+2i}=\lambda_{2c-g-2}$ and since
$c>2$
this is equal to $2c-2$.
Now $\lambda_{2c-g-2}/2=c-1$ and the largest non-gap which is at most
$c-1$ is $d$. On the other hand, $\lambda_{2c-g-2}-d=2c-2-d\geq c$
because $c\geq d+2$. Consequently $\lambda_{2c-g-2}-d\in\Lambda$ and
$\tau_{2c-g-2}=\lambda^{-1}(d)=c-g-1$.
Similarly,
the largest non-gap which is at most
$\lambda_{2c-g-1}/2$ is $d$ and
$\lambda_{2c-g-1}-d=2c-1-d\in\Lambda$.
So, $\tau_{2c-g-1}=c-g-1$.

\item
If $c=3$ then $g=2$ and $\lambda_{(2c-g-2)-2}=\lambda_0$ and $\tau_{(2c-g-2)-2}=0=c-g-1$.
Assume $c\geq 4$.
If $d=c-2$ then
$\lambda_{(2c-g-2)-2}=2c-4=2d$, so $\tau_{(2c-g-2)-2}=\lambda^{-1}(d)=c-g-1$.
If $d=c-3$ then
$\lambda_{(2c-g-2)-1}=2c-3=d+c$, so $\tau_{(2c-g-2)-1}=\lambda^{-1}(d)=c-g-1$.
Suppose now $d\leq c-4$. In this case $\lambda_{(2c-g-2)-2}/2=c-2$, which is between $d$ and $c$,
and $\lambda_{(2c-g-2)-2}-d=2c-4-d\geq c$. So $\lambda_{(2c-g-2)-2}-d\in\Lambda$.
This makes $\tau_{(2c-g-2)-2}=\lambda^{-1}(d)=c-g-1$.
\end{enumerate}
\end{proof}

\begin{lemma}
\label{lemma:trivialcase}
The trivial semigroup is the unique numerical semigroup with $\tau$ sequence
equal to $0,0,1,1,2,2,3,3,4,4,5,5,\dots$.
\end{lemma}

\begin{proof}
It is enough to check that for any other numerical semigroup there is
a value in the $\tau$ sequence that appears at least three times.
If $c=2$ then $\tau_0=\tau_1=\tau_2=0$. If $c>2$, by Lemma~\ref{lemma:precharacterization},
$\tau_{2c-g-2}=\tau_{2c-g-1}$ and they are equal to at least one of
$\tau_{2c-g-3}$ and $\tau_{2c-g-4}$.
%If a numerical semigroup is ordinary and non-trivial then $\tau_0=\tau_1=\tau_2=0$
%because $\lambda_1$ and $\lambda_2$ must be generators.
%If a numerical semigroup is non ordinary then $d>0$. In this case $2d$,$d+c$ and $d+c+1$
%are all different and at least as large as $c$ and so they belong to the numerical semigroup. On the other hand,
%it must be $\tau_{\lambda^{-1}(2d)}=\tau_{\lambda^{-1}(d+c)}=\tau_{\lambda^{-1}(d+c+1)}=\lambda^{-1}(d)=c-g-1$.
\end{proof}

\begin{theorem}
The $\tau$ sequence of a numerical semigroup
determines it.
\end{theorem}

\begin{proof}
Let $k$ be the minimum integer
such that $\tau_{k+2i}=\tau_{k+2i+1}$
and $\tau_{k+2i+2}=\tau_{k+2i+1}+1$
for all
$i\in{\mathbb N}_0$.
If $k=0$, by Lemma~\ref{lemma:trivialcase},
$\Lambda={\mathbb N}_0$. Assume $k>0$.

By Lemma~\ref{lemma:precharacterization},
if $c>2$,
$k=2c-g-2$ and $\tau_k=c-g-1$.
So,
$$\left\{\begin{array}{rcl}
c&=&k-\tau_k+1\\
g&=&k-2\tau_k\\
\end{array}\right.$$
This result can be extended to the case $c=2$
since in this case $c=2$, $g=1$, $k=1$ and $\tau_k=0$.

This determines $\lambda_i=i+g$ for all $i\geq c-g$.
Now we can determine $\lambda_{c-g-1}$, $\lambda_{c-g-2}$, and so on using that
the smallest $j$ for which $\tau_j=i$ corresponds to $\lambda_j=2\lambda_i$.
That is, $\lambda_i=\frac{1}{2}\min\{\lambda_j:\tau_j=i\}$.
\end{proof}

We have just seen that any numerical semigroup
is uniquely determined by its $\tau$ sequence.
The next lemma shows that no finite subset of $\tau$
can determine the numerical semigroup.
This result is analogous to \cite[Proposition 2.2.]{Bras:2007_IEEE_ANote}.
In this case it refers to the $\nu$ sequence
instead of the $\tau$ sequence.

\begin{lemma}
\label{lemma:tau2}
Let $r$ be a positive integer.
Let $\Lambda$ be a numerical semigroup with enumeration $\lambda$
and let $m$ be an integer with $m\geq 2$.
Define the semigroup $\Lambda'=m\Lambda\cup\{i\in{\mathbb N}_0: i\geq
m\lambda_r\}$
and let $\tau^{\Lambda}$ and $\tau^{\Lambda'}$ be the $\tau$ sequence corresponding to
$\Lambda$ and $\Lambda'$ respectively.
Then $\tau_i^{\Lambda'}=\tau_i^{\Lambda}$
for all $i\leq r$ and
$\Lambda'\neq \Lambda$.
\end{lemma}

\begin{proof}
It is obvious that $\Lambda'\neq \Lambda$.
Let $\lambda'$ be the enumeration of $\Lambda'$.
For all $i\leq r$, $\lambda'_i=m\lambda_i$.
In particular, if $j\leq i\leq r$, then
it exists $k$ with $j\leq k\leq i$
and $\lambda_j+\lambda_k=\lambda_i$
if and only if
it exists $k$ with $j\leq k\leq i$
and $\lambda'_j+\lambda'_k=\lambda'_i$.
Hence, by the definition of the $\tau$ sequence,
$\tau_i^{\Lambda'}=\tau_i^{\Lambda}$.
\end{proof}

As a consequence of Lemma~\ref{lemma:tau2},
although the sequence $\tau$ of a numerical semigroup
uniquely determines it, any subset $(\tau_i)_{0\leq i\leq r-1}$ is
exactly the set of the first $r$ values of the $\tau$ sequence of
infinitely many semigroups. In fact, by varying $m$ among
the positive integers, we get an infinite set of semigroups, all of
them sharing the first $r$ values in the $\tau$ sequence.

\subsection{Counting}

We are interested on the number $n_g$ of numerical semigoups of genus
$g$.
It is obvious that $n_0=1$ since ${\mathbb N}_0$ is the unique
numerical semigroup of genus $0$.
On the other hand, if $1$ is in a numerical semigroup, then any
non-negative integer must belong also
to the numerical semigroup, because any non-negative
integer is a finite sum of $1$'s. Thus, the unique numerical
semigroup with genus $1$ is $\{0\}\cup\{i\in{\mathbb N}_0:i\geq 2\}$
and $n_1=1$.
In \cite{Komeda89} all terms of the sequence $n_g$
are computed up to genus $37$
and the terms of genus up to $50$ are computed in \cite{Bras:Fibonacci}.
Recently we computed $n_{51}=164253200784$ and $n_{52}=266815155103$.
It is conjectured in \cite{Bras:Fibonacci}
that the sequence given by the numbers $n_g$ of numerical semigroups of genus $g$ 
asymptotically behaves like the Fibonacci numbers and so it increases
by a portion of the golden ratio.
More precisely, it is conjectured that 
(1) $n_g\geq n_{g-1}+n_{g-2}$, (2) $\lim_{g\to\infty}\frac{n_{g-1}+n_{g-2}}{n_g}=1$, 
(3) $\lim_{g\to\infty}\frac{n_{g}}{n_{g-1}}=\phi$, where $\phi=\frac{1+\sqrt{5}}{2}$ is the golden ratio. Notice that (2) and (3) are equivalent.
By now, only some bounds are known for $n_g$ which become very poor when 
$g$ approaches infinity \cite{Bras:2009_JPAA,Elizalde}.
Other contributions related to this sequence can be found in 
\cite{Komeda89,Komeda98,Sloane,Medeiros,BAdM,BrasBulygin,Zhao,Kaplan,BlancoGarciaPuerto}.

\begin{table}
\begin{center}
\resizebox{.65\textwidth}{!}{%
\begin{tabular}{|ccccc|}
\hline $g$ & $n_g$ & $n_{g-1}+n_{g-2}$ & $\frac{n_{g-1}+n_{g-2}}{n_g}$ & $\frac{n_{g}}{n_{g-1}} $\\\hline 
0 & 1 & & & \\
1 & 1 & & & 1 \\
2 & 2 & 2 &
1
&
2
\\
3 & 4 & 3 &
0.75
&
2
\\
4 & 7 & 6 &
0.857143
&
1.75
\\
5 & 12 & 11 &
0.916667
&
1.71429
\\
6 & 23 & 19 &
0.826087
&
1.91667
\\
7 & 39 & 35 &
0.897436
&
1.69565
\\
8 & 67 & 62 &
0.925373
&
1.71795
\\
9 & 118 & 106 &
0.898305
&
1.76119
\\
10 & 204 & 185 &
0.906863
&
1.72881
\\
11 & 343 & 322 &
0.938776
&
1.68137
\\
12 & 592 & 547 &
0.923986
&
1.72595
\\
13 & 1001 & 935 &
0.934066
&
1.69088
\\
14 & 1693 & 1593 &
0.940933
&
1.69131
\\
15 & 2857 & 2694 &
0.942947
&
1.68754
\\
16 & 4806 & 4550 &
0.946733
&
1.68218
\\
17 & 8045 & 7663 &
0.952517
&
1.67395
\\
18 & 13467 & 12851 &
0.954259
&
1.67396
\\
19 & 22464 & 21512 &
0.957621
&
1.66808
\\
20 & 37396 & 35931 &
0.960825
&
1.66471
\\
21 & 62194 & 59860 &
0.962472
&
1.66312
\\
22 & 103246 & 99590 &
0.964589
&
1.66006
\\
23 & 170963 & 165440 &
0.967695
&
1.65588
\\
24 & 282828 & 274209 &
0.969526
&
1.65432
\\
25 & 467224 & 453791 &
0.971249
&
1.65197
\\
26 & 770832 & 750052 &
0.973042
&
1.64981
\\
27 & 1270267 & 1238056 &
0.974642
&
1.64792
\\
28 & 2091030 & 2041099 &
0.976121
&
1.64613
\\
29 & 3437839 & 3361297 &
0.977735
&
1.64409
\\
30 & 5646773 & 5528869 &
0.979120
&
1.64254
\\
31 & 9266788 & 9084612 &
0.980341
&
1.64108
\\
32 & 15195070 & 14913561 &
0.981474
&
1.63973
\\
33 & 24896206 & 24461858 &
0.982554
&
1.63844
\\
34 & 40761087 & 40091276 &
0.983567
&
1.63724
\\
35 & 66687201 & 65657293 &
0.984556
&
1.63605
\\
36 & 109032500 & 107448288 &
0.985470
&
1.63498
\\
37 & 178158289 & 175719701 &
0.986312
&
1.63399
\\
38 & 290939807 & 287190789 &
0.987114
&
1.63304
\\
39 & 474851445 & 469098096 &
0.987884
&
1.63213
\\
40 & 774614284 & 765791252 &
0.988610
&
1.63128
\\
41 & 1262992840 & 1249465729 &
0.989290
&
1.63048
\\
42 & 2058356522 & 2037607124 &
0.989919
&
1.62975
\\
43 & 3353191846 & 3321349362 &
0.990504
&
1.62906
\\
44 & 5460401576 & 5411548368 &
0.991053
&
1.62842
\\
45 & 8888486816 & 8813593422 &
0.991574
&
1.62781
\\
46 & 14463633648 & 14348888392 &
0.992067
&
1.62723
\\
47 & 23527845502 & 23352120464 &
0.992531
&
1.62669
\\
48 & 38260496374 & 37991479150 &
0.992969
&
1.62618
\\
49 & 62200036752 & 61788341876 &
0.993381
&
1.62570
\\
50 & 101090300128 & 100460533126 &
0.993770
&
1.62525
\\
51 & 164253200784 & 163290336880 &
0.994138
&
1.62482
\\
52 & 266815155103 & 265343500912 &
0.994484
&
1.62441
\\
\hline
\end{tabular}

}
\end{center}
\caption{Computational results on the number of numerical semigroups
  up to genus $52$.}
\label{taula}
\end{table}

In Table~\ref{taula} there are the results obtained for all numerical semigroups
with genus up to $52$. For each genus we wrote the number of numerical semigroups of
the given genus, the Fibonacci-like-estimated value given by the sum of the number of semigroups
of the two previous genus, the value of the quotient
$\frac{n_{g-1}+n_{g-2}}{n_g}$,
and the value of the quotient $\frac{n_g}{n_{g-1}}$.
In Figure~\ref{graficfib} and Figure~\ref{graficor}
we depicted the behavior of these quotients. From
these graphics one can predict that
$\frac{n_{g-1}+n_{g-2}}{n_g}$
approaches $1$ as $g$ approaches infinity
whereas
$\frac{n_g}{n_{g-1}}$
approaches the golden ratio as $g$ approaches infinity.

\begin{figure}
\begin{center}
\resizebox{.9\textwidth}{!}{%\input{grafic}
\def\simbol{\circle*{2}}
\begin{picture}(350.000004,150.000000)
\put(20.000000,20.000000){\vector(1,0){315.000004}}
\put(20.000000,20.000000){\vector(0,1){115.000000}}
\put(342.500004,20.000000){\makebox(0,0){\phantom{mm}\huge{$g$}}}
\put(20.000000,142.500000){\makebox(0,0){\huge{$\frac{n_{g-1}+n_{g-2}}{n_g}$}}}
\put(10.000000,120.000000){\makebox(0,0){\huge{1}}}
{\thinlines\put(20.000000,120.000000){\line(1,0){315.000004}}}
\put(10.000000,20.000000){\makebox(0,0){\huge{0}}}
\put(19.000000,120.000000){\line(1,0){2.000000}}
\put(320.000004,10.000000){\makebox(0,0){\huge{52}}}
\put(320.000004,19.000000){\line(0,1){2.000000}}
\put(31.538462,120.000000){\simbol}
\put(37.307693,95.000000){\simbol}
\put(43.076923,105.714298){\simbol}
\put(48.846154,111.666698){\simbol}
\put(54.615385,102.608700){\simbol}
\put(60.384616,109.743602){\simbol}
\put(66.153847,112.537302){\simbol}
\put(71.923078,109.830500){\simbol}
\put(77.692308,110.686297){\simbol}
\put(83.461539,113.877602){\simbol}
\put(89.230770,112.398602){\simbol}
\put(95.000001,113.406600){\simbol}
\put(100.769232,114.093299){\simbol}
\put(106.538463,114.294697){\simbol}
\put(112.307693,114.673300){\simbol}
\put(118.076924,115.251697){\simbol}
\put(123.846155,115.425898){\simbol}
\put(129.615386,115.762098){\simbol}
\put(135.384617,116.082503){\simbol}
\put(141.153848,116.247202){\simbol}
\put(146.923079,116.458900){\simbol}
\put(152.692309,116.769500){\simbol}
\put(158.461540,116.952599){\simbol}
\put(164.230771,117.124898){\simbol}
\put(170.000002,117.304201){\simbol}
\put(175.769233,117.464198){\simbol}
\put(181.538464,117.612101){\simbol}
\put(187.307694,117.773498){\simbol}
\put(193.076925,117.912002){\simbol}
\put(198.846156,118.034102){\simbol}
\put(204.615387,118.147398){\simbol}
\put(210.384618,118.255402){\simbol}
\put(216.153849,118.356700){\simbol}
\put(221.923079,118.455602){\simbol}
\put(227.692310,118.547000){\simbol}
\put(233.461541,118.631197){\simbol}
\put(239.230772,118.711401){\simbol}
\put(245.000003,118.788399){\simbol}
\put(250.769234,118.861003){\simbol}
\put(256.538465,118.929000){\simbol}
\put(262.307695,118.991901){\simbol}
\put(268.076926,119.050403){\simbol}
\put(273.846157,119.105299){\simbol}
\put(279.615388,119.157399){\simbol}
\put(285.384619,119.206698){\simbol}
\put(291.153850,119.253100){\simbol}
\put(296.923080,119.296898){\simbol}
\put(302.692311,119.338102){\simbol}
\put(308.461542,119.377000){\simbol}
\put(314.230773,119.413800){\simbol}
\put(320.000004,119.448401){\simbol}
\put(320.000004,119.448401){\simbol}
\end{picture}

}
\end{center}
\caption{Behavior of the quotient $\frac{n_{g-1}+n_{g-2}}{n_g}$.
The values in this graphic correspond to
the values in Table~\ref{taula}.}
\label{graficfib}
\end{figure}
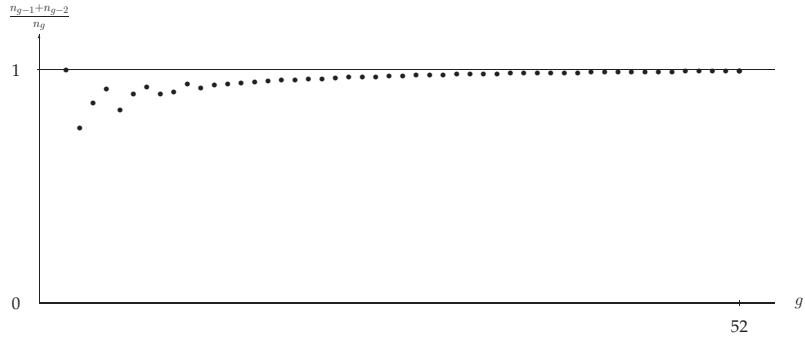

\begin{figure}
\begin{center}
\resizebox{.9\textwidth}{!}
{%\input{graficor}
\def\simbol{\circle*{2}}
\begin{picture}(350.000004,150.000000)
\put(20.000000,20.000000){\vector(1,0){315.000004}}
\put(20.000000,20.000000){\vector(0,1){115.000000}}
\put(342.500004,20.000000){\makebox(0,0){\phantom{mm}\huge{$g$}}}
\put(20.000000,142.500000){\makebox(0,0){\huge{$\frac{n_g}{n_{g-1}}$}}}
\put(10.000000,100.901699){\makebox(0,0){\huge{$\phi$}}}
{\thinlines\put(20.000000,100.901699){\line(1,0){315.000004}}}
\put(10.000000,20.000000){\makebox(0,0){\huge{0}}}
\put(19.000000,100.901699){\line(1,0){2.000000}}
\put(320.000004,10.000000){\makebox(0,0){\huge{52}}}
\put(320.000004,19.000000){\line(0,1){2.000000}}
\put(25.769231,70.000000){\simbol}
\put(31.538462,120.000000){\simbol}
\put(37.307693,120.000000){\simbol}
\put(43.076923,107.500000){\simbol}
\put(48.846154,105.714501){\simbol}
\put(54.615385,115.833498){\simbol}
\put(60.384616,104.782499){\simbol}
\put(66.153847,105.897499){\simbol}
\put(71.923078,108.059503){\simbol}
\put(77.692308,106.440498){\simbol}
\put(83.461539,104.068501){\simbol}
\put(89.230770,106.297500){\simbol}
\put(95.000001,104.543997){\simbol}
\put(100.769232,104.565502){\simbol}
\put(106.538463,104.377003){\simbol}
\put(112.307693,104.109002){\simbol}
\put(118.076924,103.697498){\simbol}
\put(123.846155,103.697999){\simbol}
\put(129.615386,103.403999){\simbol}
\put(135.384617,103.235502){\simbol}
\put(141.153848,103.156002){\simbol}
\put(146.923079,103.003002){\simbol}
\put(152.692309,102.793999){\simbol}
\put(158.461540,102.716000){\simbol}
\put(164.230771,102.598501){\simbol}
\put(170.000002,102.490498){\simbol}
\put(175.769233,102.396001){\simbol}
\put(181.538464,102.306498){\simbol}
\put(187.307694,102.204503){\simbol}
\put(193.076925,102.126999){\simbol}
\put(198.846156,102.054001){\simbol}
\put(204.615387,101.986499){\simbol}
\put(210.384618,101.922001){\simbol}
\put(216.153849,101.862003){\simbol}
\put(221.923079,101.802499){\simbol}
\put(227.692310,101.748998){\simbol}
\put(233.461541,101.699502){\simbol}
\put(239.230772,101.651998){\simbol}
\put(245.000003,101.606501){\simbol}
\put(250.769234,101.563997){\simbol}
\put(256.538465,101.524003){\simbol}
\put(262.307695,101.487501){\simbol}
\put(268.076926,101.453001){\simbol}
\put(273.846157,101.421000){\simbol}
\put(279.615388,101.390500){\simbol}
\put(285.384619,101.361502){\simbol}
\put(291.153850,101.334502){\simbol}
\put(296.923080,101.309003){\simbol}
\put(302.692311,101.285000){\simbol}
\put(308.461542,101.262499){\simbol}
\put(314.230773,101.241000){\simbol}
\put(320.000004,101.220502){\simbol}
\put(320.000004,101.220502){\simbol}
\end{picture}

}
\end{center}
\caption{Behavior of the quotient $\frac{n_{g}}{n_{g-1}}$.
The values in this graphic correspond to
the values in Table~\ref{taula}.}
\label{graficor}
\end{figure}

The number $n_g$ 
is usually studied by means of the tree rooted at the semigroup ${\mathbb N}_0$
and for which the children of a semigroup are the semigroups obtained by taking out
one by one its generators larger than or equal to its conductor \cite{Bras:Fibonacci,Bras:2009_JPAA,BrasBulygin,Elizalde}.
This tree was 
previously used in \cite{
%Rosales:families,
RoGaGaJi:fundamentalgaps,RoGaGaJi:oversemigroups}.
It is illustrated in Figure~\ref{fig:tree}. 
It contains all semigroups exactly once and the semigroups at depth $g$ 
have genus $g$. So, $n_g$ is the number of nodes of the tree at depth $g$.
Some alternatives 
for counting semigroups of a given genus without using this tree
have been considered in \cite{BAdM,BlancoGarciaPuerto,Zhao}.

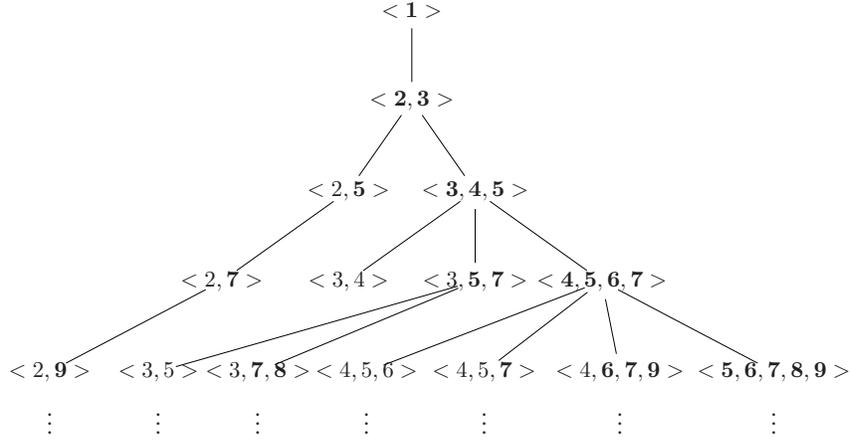
\begin{figure}
\compatiblegastexun
\setvertexdiam{4}
\letvertex S1=(40,50)
\letvertex S23=(40,40)
\letvertex S345=(47,30)
\letvertex S25=(33,30)
\letvertex S4567=(61,20)
\letvertex S357=(47,20)
\letvertex S34=(33,20)
\letvertex S27=(19,20)
\letvertex S56789=(80,10)
\letvertex P56789=(80,5)
\letvertex S4679=(63,10)
\letvertex P4679=(63,5)
\letvertex S457=(48,10)
\letvertex P457=(48,5)
\letvertex S456=(35,10)
\letvertex P456=(35,5)
\letvertex S378=(23,10)
\letvertex P378=(23,5)
\letvertex S35=(12,10)
\letvertex P35=(12,5)
\letvertex S29=(0,10)
\letvertex P29=(0,5)
\begin{center}
\resizebox{.8\textwidth}{!}{
\begin{picture}(80,50)
\drawvertex(S1){$<{\bf 1}>$}
\drawvertex(S23){$<{\bf 2},{\bf 3}>$}
\drawvertex(S345){$<{\bf 3},{\bf 4},{\bf 5}>$}
\drawvertex(S25){$<2,{\bf 5}>$}
\drawvertex(S4567){$<{\bf 4},{\bf 5},{\bf 6},{\bf 7}>$}
\drawvertex(S357){$<3,{\bf 5},{\bf 7}>$}
\drawvertex(S34){$<3,4>$}
\drawvertex(S27){$<2,{\bf 7}>$}
\drawvertex(S56789){$<{\bf 5},{\bf 6},{\bf 7},{\bf 8},{\bf 9}>$}
\drawvertex(P56789){$\vdots$}
\drawvertex(S4679){$<4,{\bf 6},{\bf 7},{\bf 9}>$}
\drawvertex(P4679){$\vdots$}
\drawvertex(S457){$<4,5,{\bf 7}>$}
\drawvertex(P457){$\vdots$}
\drawvertex(S456){$<4,5,6>$}
\drawvertex(P456){$\vdots$}
\drawvertex(S378){$<3,{\bf 7},{\bf 8}>$}
\drawvertex(P378){$\vdots$}
\drawvertex(S35){$<3,5>$}
\drawvertex(P35){$\vdots$}
\drawvertex(S29){$<2,{\bf 9}>$}
\drawvertex(P29){$\vdots$}
\drawundirectededge(S1,S23){}
\drawundirectededge(S23,S345){}
\drawundirectededge(S23,S25){}
\drawundirectededge(S345,S4567){}
\drawundirectededge(S345,S357){}
\drawundirectededge(S345,S34){}
\drawundirectededge(S25,S27){}
\drawundirectededge(S4567,S56789){}
\drawundirectededge(S4567,S4679){}
\drawundirectededge(S4567,S457){}
\drawundirectededge(S4567,S456){}
\drawundirectededge(S357,S378){}
\drawundirectededge(S357,S35){}
\drawundirectededge(S27,S29){}
\end{picture}}\end{center}
\caption{Recursive construction of numerical semigroups of genus $g$ from numerical
semigroups of genus $g-1$.
Generators larger than the conductor are written in bold face.}
\label{fig:tree}
\end{figure}

%\subsubsection{Dyck paths and Catalan bounds}
%\subsubsection{Semigroup tree and Fibonacci bounds}
%\subsubsection{Generating trees and further bounds}

\section{Numerical semigroups and codes}

\subsection{One-point codes and their decoding}

\subsubsection{One-point codes}

\paragraph{Linear codes}

A \alert{linear code} $C$ of length $n$ over the alphabet ${\mathbb F}_q$
is a vector subspace of ${\mathbb F}_q^n$. 
Its elements are called \alert{code words}.
The \alert{dimension} $k$
of the code is the dimension of $C$ as a subspace of ${\mathbb F}_q^n$. 
The \alert{dual code} of $C$ is $C^\perp=\{v\in{\mathbb F}_q^n:v\cdot c=0 \mbox{ for all }c\in C\}$. It is a linear code with the same length as $C$ and with dimension $n-k$.

The \alert{Hamming distance} between two vectors of the same length is the number of positions in which they do not agree. The \alert{weight} of a vector
is the number of its non-zero components or, equivalently, its Hamming distance to the zero vector.
The \alert{minimum distance} $d$ of a linear code $C$ is the minimum Hamming distance between two code words in $C$. Equivalently, it is the minimum weight of all code words in $C$. 
The \alert{correction capability}
of a code is the maximum number of errors that can be added to any code word,
with the code word being still uniquelly identifiable.
The correction capability of a linear code with minimum distance $d$ is 
$\lfloor\frac{d-1}{2}\rfloor$.

\paragraph{One-point codes}
Let $P$ be a rational point of the 
algebraic smooth curve ${\mathcal X}_F$ defined over ${\mathbb F}_q$
with Weierstrass semigroup
$\Lambda$.
Suppose that the enumeration of $\Lambda$ is $\lambda$.
Recall that $A=\bigcup_{m\geq 0}L(m P)$ is the ring of
rational functions having poles only at $P$.
We will say that the \alert{order} of $f\in A\setminus\{0\}$ is $s$ if 
$v_P(f)=-\lambda_s$. The order of $0$ is considered to be either $-1$ \cite{O'Sullivan:newcodes} or $-\infty$ \cite{HoLiPe:agc}. In the present work we will consider the order of $0$ to be $-1$ although both would be fine.
We denote the order of $f$ by $\rho(f)$.

One can find an infinite basis
$z_0,z_1,\dots,z_i,\dots$ of $A$
such that $v_P(z_i)=-\lambda_i$ or, equivalently, $\rho(z_i)=i$.
Consider a set of rational points $P_1,\dots,P_n$ different from $P$
and the map $\varphi:A\longrightarrow{{\mathbb F}_q}^n$ defined by
$\varphi(f)=(f(P_1),\dots,f(P_n))$.
To each finite subset $W\subseteq{\mathbb N}_0$ we associate
the \alert{one-point} code $$C_W=<\varphi(z_i):i\in W>^\perp=<(z_i(P_1),\dots,z_i(P_n)):i\in W>^\perp.$$
We say that $W$ is the set of \alert{parity checks} of $C_W$.
The one-point codes for which $W=\{0,1,\dots,m\}$ are called \alert{classical} 
one-point codes. In this case we write $C_m$ for $C_W$.

\subsubsection{Decoding one-point codes}

This section presents a sketch of a decoding algorithm for $C_W$.
The aim is to justify the conditions guaranteeing correction of errors.
Suppose that a code word $c\in C_W$ is sent and that an error $e$ is added to it so that the received word is $u=c+e$. We will use $t$ for the number of non-zero positions in $e$.

\begin{definition}
A polynomial $f$ is an \alert{error-locator} of an error vector $e$ 
if and only if $f(P_i)=0$ whenever $e_i\neq 0$.
The \alert{footprint} of $e$ is the set $$\Delta_e={\mathbb N}_0\setminus\{\rho(f):f\mbox{ is an error-locator}\}.$$
\end{definition}
It is well known that $\#\Delta_e=t$ and that $\Delta_e$ is $\preceq$-closed. That is, 
if $i\preceq j$ and  $j\in \Delta_e$ then $i\in\Delta_e$.
If for each minimal element in ${\mathbb N}_0\setminus\Delta_e$ with respect to $\preceq$
we can find an error-locator with that order then localization of errors is guaranteed.
%So, error-locating is based on finding error-locators.

\begin{definition}
Define the \alert{syndrome} of orders $i,j$ as $$s_{ij}=\sum_{l=1}^nz_i(P_l)z_j(P_l)e_l.$$
The \alert{syndrome matrix} $S^{rr'}$ is the $(r+1)\times (r'+1)$
matrix with $S^{rr'}_{ij}=s_{ij}$ for $0\leq i\leq r$ and $0\leq j\leq r'$.
\end{definition}
%Notice that for the code $C_W$, $s_{ij}$ can be directly computed from $e$ if $i\oplus j\in W$.
%The other syndromes are {\it unknown}.
The matrix $S^{r'r}$ is the matrix $S^{r'r}$, transposed.
By Lemma~\ref{lemma:propsv}, 
if
$i\oplus j=k$ then there exist $a_0,\dots,a_k$ such that
$z_iz_j=a_kz_k+\dots+a_0z_0$. Define
$s_{k}=\sum_{l=1}^nz_k(P_l)e_l.$
Then 
\begin{equation}
\label{eqnsynds}
s_{ij}=a_ks_k+\dots+a_0s_0.
\end{equation}

The syndromes depend on $e$ which is initially unknown. So, in general, $s_{ij}$ and $s_k$ are unknown.
For a polynomial $f=z_r+a_{r-1}z_{r-1}+\dots+a_0z_0$,
being an error locator means that $(a_0,\dots,a_r)S^{rr'}=0$ 
for all
$r'>0$.\footnote{$(a_0,\dots,a_r)S^{rr'}=(\sum_{i=0}^{r}a_is_{i0},\dots,\sum_{i=0}^{r}a_is_{ir'})=(\sum_{i=0}^{r}a_i\sum_{k=1}^nz_i(P_k)z_0(P_k)e_k,\dots,\sum_{i=0}^{r}a_i\sum_{k=1}^nz_i(P_k)z_{r'}(P_k)e_k)
=(\sum_{k=1}^nz_0(P_k)e_k\sum_{i=0}^{r}a_iz_i(P_k),\dots,\sum_{k=1}^nz_{r'}(P_k)e_k\sum_{i=0}^{r}a_iz_i(P_k))
=(\sum_{k=1}^nz_0(P_k)e_kf(P_k),\dots,\sum_{k=1}^nz_{r'}(P_k)e_kf(P_k))=(0,\dots,0)$ }
Conversely, there exists $M$ such that if 
$(a_0,\dots,a_r)S^{rr'}=0$ 
for all
$r'$ with 
$r\oplus r'\leq M$ 
then $f$ is an error locator.
\footnote{An explanation for this can be found in \cite{Pretzel,HoLiPe:agc,BrOS:2006_AAECC}.}

Hence, we look for pairs $(r,r')$ with $r\oplus r'$ large enough
such that $(x_0,\dots,x_{r-1},1)S^{rr'}=0$ has a non-zero solution 
(and indeed we look for this solution).
Notice that if $(x_0,\dots,x_{r-1},1)S^{rr'}=0$ has a non-zero solution then so does 
$(x_0,\dots,x_{r-1},1)S^{r{r''}}$ for all $r''<r'$.

The first difficulty is that only a few syndromes are known. 
This is overcome by using a majority voting procedure.

We proceed iteratively, considering the non-gaps of 
$\Lambda$ by increasing order. 
Suppose that all syndromes $s_{ij}$ 
are known for $i\oplus j< k$ and we want to compute the syndromes 
$s_{ij}$ with $i\oplus j=k$.
By equation~\ref{eqnsynds} this is equivalent to finding $s_k$.
If $k\in W$ then the computation can be done by just
using the definition of $C_W$:
%\begin{eqnarray*}
%s_k&=&\sum_{l=1}^nz_k(P_l)e_l\\%&=&\sum_{l=1}^nz_k(P_l)(u_l-c_l)\\
%&=&
%\sum_{l=1}^nz_k(P_l)u_l
%-\sum_{l=1}^nz_k(P_l)c_l\\&=&\sum_{l=1}^nz_k(P_l)u_l.
%\end{eqnarray*}
$s_k=\sum_{l=1}^nz_k(P_l)e_l=\sum_{l=1}^nz_k(P_l)u_l-\sum_{l=1}^nz_k(P_l)c_l=\sum_{l=1}^nz_k(P_l)u_l$.
Otherwise we establish a voting procedure to determine $s_{k}$ 

In the voting procedure the voters are 
the elements $i\preceq k$ for which 
%\begin{itemize}
%\item 
$(x_0,\dots,x_{i-1},1)S^{i,k\ominus i-1}=0$ 
%has a non-zero solution
%\item 
and
$(y_0,\dots,y_{k\ominus i-1},1)S^{k\ominus i,i-1}=0$ 
%has 
have
%a non-zero solution,
non-zero solutions.
We consider the value
\begin{center}
\resizebox{.97\textwidth}{!}{
$\tilde s_{i,k\ominus i}=(s_{0,(k\ominus i)},\dots,s_{i-1,(k\ominus i)})\cdot(x_0,\dots,x_{i-1})
=(s_{i,0},\dots,s_{i,(k\ominus i)-1})\cdot(y_0,\dots,y_{(k\ominus i)-1}).$}
%$\tilde s_{i,k\ominus i}=(s_{0,j},\dots,s_{i-1,j})\cdot(x_0,\dots,x_{i-1})
%=(s_{i,0},\dots,s_{i,j-1})\cdot(y_0,\dots,y_{j-1}).$}
\footnote{
%\begin{lemma}
If 
%\begin{itemize}
%\item 
$(x_0,\dots,x_{i-1},1)S^{i,j-1}=0$ 
and
%\item 
$(y_0,\dots,y_{j-1},1)S^{j,i-1}=0$
%\end{itemize}
then $(s_{0,j},\dots,s_{i-1,j})\cdot(x_0,\dots,x_{i-1})
=(s_{i,0},\dots,s_{i,j-1})\cdot(y_0,\dots,y_{j-1}).$
%\end{lemma}
%\begin{proof}
Indeed,
$(x_0,\dots,x_{i-1},1)S^{i,j-1}=0$ 
implies that
$(s_{i,0},\dots,s_{i,j-1})=-(x_0,\dots,x_{i-1})S^{i-1,j-1}$
and similarly,
 $(y_0,\dots,y_{j-1},1)S^{j,i-1}=0$
implies that
$(s_{0,j},\dots,s_{i-1,j})=-(y_0,\dots,y_{j-1})S^{j-1,i-1}$.
Now,
$(s_{0,j},\dots,s_{i-1,j})\cdot (x_0,\dots,x_{i-1})=
-(y_0,\dots,y_{j-1})S^{j-1,i-1}\cdot (x_0,\dots,x_{i-1})=
-(y_0,\dots,y_{j-1})S^{j-1,i-1} (x_0,\dots,x_{i-1})^T=
-(x_0,\dots,x_{i-1})S^{i-1,j-1} (y_0,\dots,y_{j-1})^T=
-(x_0,\dots,x_{i-1})S^{i-1,j-1}\cdot (y_0,\dots,y_{j-1})=
(s_{i,0},\dots,s_{i,j-1})\cdot (y_0,\dots,y_{j-1})$
%\end{proof}
}
\end{center}
as a candidate for $s_{i,k\ominus i}$.
Notice that if $s_{i,k\ominus i}=\tilde s_{i,k\ominus i}$ 
then $(x_0,\dots,x_{i-1},1)S^{i,k\ominus i}=0$ 
and
$(y_0,\dots,y_{k\ominus i-1},1)S^{k\ominus i,i}=0$. 
Otherwise, if $s_{i,k\ominus i}\neq \tilde s_{i,k\ominus i}$ then
there exist no error-locators of order $i$
and no error-locators of order $k\ominus i$.
Since $\tilde s_{i,k\ominus i}$ is a candidate for $s_{i,k\ominus i}$,
the associated candidate $\tilde s_k$ for $s_k$ will be derived from the equation
$\tilde s_{i,k\ominus i}=a_k\tilde s_k+a_{k-1} s_{k-1}+\dots+a_0s_0,$
where $a_0,\dots,a_k$ are such that $z_iz_{k\ominus i}=a_kz_k+\dots+a_0z_0$. 
That is, $\tilde s_k=\frac{\tilde s_{i,k\ominus i}-a_{k-1} s_{k-1}-\dots-a_0s_0}{a_k}.$

\begin{lemma}
\label{lemma:condicionspercalcularsindromes}
\begin{itemize}
\item
If $i\in N_k$ and $i,k\ominus i\not\in\Delta_e$ 
then $i$ is a voter and its vote coincides with $s_k$.
\item
If a voter $i$ votes for a wrong cadidate for $s_k$ then $i,k\ominus i\in\Delta_e$.
\item 
If $\nu_k>2\#(N_k\cap\Delta_e)$ 
then a majority of voters vote for the right value $s_k$.
\end{itemize}
\end{lemma}

\begin{proof}
The first two items are deduced from what has been said before.
Consider the sets 
$A=\{i\in N_k: i, k\ominus i\in \Delta_e\}$, 
$B=\{i\in N_k: i\in\Delta_e, k\ominus i\not\in \Delta_e\}$, 
$C=\{i\in N_k: i\not\in\Delta_e, k\ominus i\in \Delta_e\}$, 
$D=\{i\in N_k: i, k\ominus i\not\in \Delta_e\}$.
By the previous items, the wrong votes are at most $\#A$ while the right votes are at least 
$\#D$.
Obviously,
$\nu_k=\#A+\#B+\#C+\#D$,
$\#(N_k\cap\Delta_e)=\#A+\#B=\#A+\#C$.
So, the difference between the right and the wrong votes is at least 
$\#D-\#A=\nu_k-2\#A-\#B-\#C=\nu_k-2\#(N_k\cap\Delta_e)>0$.
\end{proof}

The conclusion of this section is the next theorem.

\begin{theorem}
\label{t:correctionrequirement}
If $\nu_i>2\#(N_i\cap\Delta_e)$ 
for all $i\not\in W$ then $e$ is correctable by $C_W$.
\end{theorem}

\subsection{The $\nu$ sequence, classical codes, and Feng-Rao improved codes}

From the equality $\#\Delta_e=t$ we deduce the next lemma.

\begin{lemma}
\label{lemma: condicio mes feble}
If
the number $t$ of errors in $e$ satisfies
$t\leq\lfloor\frac{\nu_i-1}{2}\rfloor$,
then
$\nu_i>2\#(N_i\cap\Delta_e)$.
\end{lemma}

\subsubsection{The $\nu$ sequence and the minimum distance of classical codes}

Theorem~\ref{t:correctionrequirement} and 
Lemma~\ref{lemma: condicio mes feble} can be used in order to get an estimate
of the minimum distance of a one-point code. The next definition arises from \cite{FeRa:dmin,HoLiPe:agc,KiPe}.

\begin{definition}
The \alert{order bound} on the minimum distance of the classical code $C_W$, with $W=\{0,\dots,m\}$ is 
$$d_{ORD}(C_m)=\min\{\nu_i:i>m\}.$$
\end{definition}

The order bound is also referred to as the \alert{Feng-Rao bound}.
The order bound is proved to be a lower bound
on the minimum distance for classical codes \cite{FeRa:dmin,HoLiPe:agc,KiPe}.

\begin{lemma}
$d(C_m)\geq d_{ORD}(C_m)$.
\end{lemma}

From Lemma~\ref{lemma:nu} we deduce that
$\nu_{i+1}\leq \nu_{i+2}$ and so
$d_{ORD}(C_i)=\nu_{i+1}$ for all $i\geq 2c-g-2$.

A refined version of the order bound is 
$$d^{P_1,\dots,P_n}_{ORD}(C_m)=\min\{\nu_i:i>m, C_i\neq C_{i+1}\}.$$
While $d_{ORD}$ 
only depends on the Weierstrass semigroup, 
$d^{P_1,\dots,P_n}_{ORD}$ depends also on the points $P_1,\dots,P_n$.
Since our point of view is that of numerical semigroups we will concentrate on $d_{ORD}$.

Generalized Hamming weights are a generalization of the minimum distance of a code with many applications to coding theory but also to other fields such as cryptography. For the generalized Hamming weights of one-point codes there is a generalization of the order bound based also on the associated Weierstrass semigroups. We will not discuss this topic here but the reader interested in it can see \cite{HeijnenPellikaan,FarranMunuera}.

\subsubsection{On the order bound on the minimum distance}
\label{section: frd acute}

In this section we will find a formula for the smallest $m$ for which 
$d_{ORD}(C_i)=\nu_{i+1}$
for all $i\geq m$, for the case of acute semigroups.
At the end we will use Munuera-Torres and Oneto-Tamone's results to generalize this formula.

\begin{remark}
\label{remark: c'+d geq c}
Let $\Lambda$ be a non-ordinary numerical semigroup
with conductor $c$, subconductor $c'$ and
dominant $d$.
Then, $c'+d\geq c$.
Indeed,
$c'+d\in\Lambda$
and by Remark~\ref{remark: ordinary implica d,c' son 0}
it is strictly larger than $d$.
So, it must be larger than
or equal to
$c$.
\end{remark}

\begin{theorem}
\label{theorem: ultim punt decreixement per acute semigroups}
Let $\Lambda$ be a non-ordinary acute 
%\index{semigroup!acute}
numerical semigroup
with enumeration $\lambda$, conductor $c$, subconductor $c'$
and
dominant $d$.
Let 
\begin{equation}
\label{eq:m}
m=\min\{\lambda^{-1}(c+c'-2),
\lambda^{-1}(2d)\}.
\end{equation}
Then,
\begin{enumerate}
\item $\nu_m>\nu_{m+1}$
\item $\nu_i\leq\nu_{i+1}$
for all $i>m$.
\end{enumerate}
\end{theorem}

\begin{proof}
Following the notations in 
Lemma~\ref{lemma:nu}, 
for $i\geq \lambda^{-1}(c)$,
$g(i)=g$.
Thus, for $i\geq \lambda^{-1}(c)$ we have
\begin{equation}
\label{equation: nu creixent}
\nu_i\leq\nu_{i+1}\mbox{ if and only if }
\#D(i+1)\geq \#D(i)-1.
\end{equation}

Let $l=c-d-1$.
Notice that $l$ is the number of gaps between
the conductor and the dominant.
Since $\Lambda$ is acute,
the $l$
integers before $c'$
are also gaps.
Let us call $k=\lambda^{-1}(c'+d)$.
For all $1\leq i\leq l$, both
$(c'-i)$ and $(d+i)$
are in $D(k)$ because
they are gaps and
$$(c'-i)+(d+i)=c'+d.$$
Moreover, there are no more gaps in $D(k)$ because,
if $j\leq c'-l-1$ then $c'+d-j\geq d+l+1=c$ and so
$c'+d-j\in\Lambda$.
Therefore,
$$D(k)=\{c'-i: 1\leq i \leq l\}\cup\{d+i: 1\leq i \leq l\}.$$

Now suppose that $j\geq k$.
By Remark~\ref{remark: c'+d geq c},
$\lambda_k\geq c$ and so
$\lambda_j=\lambda_k+j-k=c'+d+j-k$.
Then,
$$D(j)=A(j)\cup B(j),$$
where
{\small
\begin{eqnarray*}
A(j)&=&\casos
{\begin{array}{l}\{c'-i: 1\leq i \leq l-j+k\}
\\\cup\{d+i: j-k+1\leq i \leq l\}\end{array}}
{{\tiny if $\lambda_k\leq\lambda_{j}\leq c+c'-2$,}}
{\emptyset}{{\tiny otherwise.}}
\\
B(j)&=&\quatrecasos
{\emptyset}{{\tiny if $\lambda_k\leq\lambda_{j}\leq 2d+1$}}
{\{d+i: 1\leq i \leq \lambda_{j}-2d-1\}}{{\tiny if $2d+2\leq\lambda_{j}\leq c+d$,}}
{\{d+i: \lambda_{j}-d-c+1\leq i \leq l\}}
{{\tiny if $c+d\leq\lambda_{j}\leq 2c-2$,}}
{\emptyset}{{\tiny if $\lambda_{j}\geq 2c-1$.}}
\end{eqnarray*}}
Notice that $A(j)\cap B(j)=\emptyset$ and hence
$$\#D(j)=\#A(j)+\#B(j).$$
We have
\begin{eqnarray*}
\#A(j)&=&\casos
{2(l-j+k)}{if $\lambda_k\leq \lambda_{j}\leq c+c'-2$,}
{0}{otherwise.}\\
\#B(j)&=&\quatrecasos
{0}{if $\lambda_k\leq \lambda_{j}\leq 2d+1$,}
{\lambda_{j}-2d-1}{if $2d+2\leq\lambda_{j}\leq c+d$,}
{2c-1-\lambda_j}{if $c+d\leq\lambda_{j}\leq 2c-2$,}
{0}{if $\lambda_{j}\geq 2c-1$.}
\end{eqnarray*}
So,
\begin{eqnarray*}
\#A(j+1)&=&\casos{\#A(j)-2}{if $\lambda_k\leq
\lambda_{j}\leq c+c'-2$,}{\#A(j)}{otherwise.}
\\
\#B(j+1)&=&\quatrecasos
{\#B(j)}{if $\lambda_k\leq\lambda_{j}\leq 2d$}
{\#B(j)+1}{if $2d+1\leq\lambda_{j}\leq c+d-1$}
{\#B(j)-1}{if $c+d\leq\lambda_{j}\leq 2c-2$}
{\#B(j)}{if $\lambda_{j}\geq 2c-1$}
\end{eqnarray*}
Notice that $c+c'-2< c+d$.
Thus, for $\lambda_j \geq c+d$,
$$\#D(j+1)=\casos{\#D(j)-1}{if $c+d\leq\lambda_{j}\leq 2c-2$,}
{\#D(j)}{if $\lambda_{j}\geq 2c-1$.}$$
Hence, by (\ref{equation: nu creixent}),
$\nu_i\leq \nu_{i+1}$ for all $i\geq \lambda^{-1}(c+d)$
because $\lambda^{-1}(c+d)\geq\lambda^{-1}(c)$.
Now, let us analyze what happens if
$\lambda_{j}< c+d$.

If $c+c'-2\leq 2d$ then
$$\#D(j+1)=\trescasos
{\#D(j)-2}{if $\lambda_k\leq \lambda_{j}\leq c+c'-2$,}
{\#D(j)}{if $c+c'-1\leq\lambda_{j}\leq 2d$,}
{\#D(j)+1}{if $2d+1\leq\lambda_{j}\leq c+d-1$.}$$
and if
$2d +1\leq c+c'-2$ then
$$\#D(j+1)=\trescasos
{\#D(j)-2}{if $\lambda_k\leq \lambda_{j}\leq 2d$,}
{\#D(j)-1}{if $2d+1\leq\lambda_{j}\leq c+c'-2$,}
{\#D(j)+1}{if $c+c'-1\leq\lambda_{j}\leq c+d-1$.}$$

So, by (\ref{equation: nu creixent})
and since both $c+c'-2$ and $2d$ are larger than or equal to $c$,
the result follows.
\end{proof}

\begin{corollary}
\label{corollary:dFR-acute}
Let $\Lambda$ be a non-ordinary acute numerical semigroup
with enumeration $\lambda$, conductor $c$ and subconductor $c'$.
Let $$m=\min\{\lambda^{-1}(c+c'-2),
\lambda^{-1}(2d)\}.$$
Then, $m$ is the smallest integer
for which $$d_{ORD}(C_i)=\nu_{i+1}$$
for all $i\geq m$.
\end{corollary}

\begin{example}
Recall the
Weierstrass semigroup at the point $P_0$
on the Klein quartic that we presented
in Example~\ref{example:klein}.
Its conductor is $5$,
its dominant is $3$
and its subconductor is $3$.
In Table~\ref{table:klein} we have,
for each integer from $0$ to
$\lambda^{-1}(2c-2)$,
the values $\lambda_i$, $\nu_i$ and $d_{ORD}(C_i)$.
%Recall that, as mentioned after the definition of $d_{ORD}$, $\nu_{i+1}\leq \nu_{i+2}$ and $d_{ORD}(C_i)=\nu_{i+1}$ for all $i\geq \lambda^{-1}(2c-1)-1$.

For this example,
$\lambda^{-1}(c+c'-2)=\lambda^{-1}(2d)=3$
and so,
$m=\min\{\lambda^{-1}(c+c'-2),
\lambda^{-1}(2d)\}=3$.
We can check that, as
stated in
Theorem~\ref{theorem: ultim punt decreixement per acute semigroups},
$\nu_3>\nu_4$
and $\nu_i\leq \nu_{i+1}$ for all $i>3$.
Moreover,
as stated in Corollary~\ref{corollary:dFR-acute},
$d_{ORD}(C_i)=\nu_{i+1}$
for all $i\geq 3$
while $d_{ORD}(C_2)\neq\nu_3$.
\end{example}

\begin{table}
\caption{Klein quartic}
\label{table:klein}
\centering
\begin{tabular}{|c|c|c|c|}
\hline
$i$ & $\lambda_i$ & $\nu_i$ & $d_{ORD}(C_i)$ \\ \hline
$0$ & $0$ & $1$ & $2$ \\
$1$ & $3$ & $2$ & $2$ \\
$2$ & $5$ & $2$ & $2$ \\
$3$ & $6$ & $3$ & $2$ \\
$4$ & $7$ & $2$ & $4$ \\
$5$ & $8$ & $4$ & $4$ \\
\hline \end{tabular}

\end{table}

\begin{lemma}
\label{lemma: qui es el minim dels dos}
Let $\Lambda$ be a non-ordinary numerical semigroup
with conductor $c$, subconductor $c'$ and dominant~$d$.
\begin{enumerate}
\item If $\Lambda$ is symmetric then $\min\{c+c'-2,2d\}=c+c'-2=2c-2-\lambda_1$,
\item If $\Lambda$ is pseudo-symmetric then $\min\{c+c'-2,2d\}=c+c'-2$,
\item If $\Lambda$ is Arf then $\min\{c+c'-2,2d\}=2d$,
\item If $\Lambda$ is generated by an interval then $\min\{c+c'-2,2d\}=c+c'-2$.
\end{enumerate}
\end{lemma}

\begin{proof}\mbox{}
\begin{enumerate}
\item We already saw in the proof
of Lemma~\ref{lemma: exemples acute}
that if
$\Lambda$ is symmetric then
$d=c-2$. So, $c+c'-2=d+c'\leq 2d$ because
$c'\leq d$.
Moreover, by Lemma~\ref{lemma:symmetric-implicacio},
any non-negative integer $i$ is a gap if and only if $c-1-i\in\Lambda$.
This implies that
$c'-1=c-1-\lambda_1$ and so
$c'=c-\lambda_1$. Therefore,
$c+c'-2=2c-2-\lambda_1$.
\item If $\Lambda$ is pseudo-symmetric and non-ordinary then
$d=c-2$ because $1$ is a gap different from $(c-1)/2$. 
So, $c+c'-2=d+c'\leq 2d$.
\item If $\Lambda$ is Arf then
$c'=d$. Indeed, if $c'<d$ then
$d-1\in\Lambda$ and, by Lemma~\ref{lemma: arf te els non-gaps separats},
$d-1\geq c$, a contradiction.
Since $d\leq c-2$, we have $2d\leq c+c'-2$.
\item Suppose $\Lambda$ is generated by the interval
$\{i,i+1,\dots,j\}$.
By Lemma~\ref{lemma: forma semigrup generat per intervals},
there exists $k$ such that
$c=k i$ and $d=(k-1) j$.
We have that $c-d\leq j-i$, because otherwise
$(k+1)i-k j=c-d-(j-i)>1$, and hence
$k j+1$ would be a gap greater than $c$.
On the other hand
$d-c'\geq j-i$,
and hence
$2d-(c+c'-2)=d-c+d-c'+2\geq i-j+j-i+2=2$.
\end{enumerate}
\end{proof}

\begin{example}
Consider the Hermitian curve
%\index{Hermitian!semigroup}
over ${\mathbb F}_{16}$.
Its numerical semigroup is generated by $4$ and $5$.
So, this is a symmetric numerical semigroup
because it is generated by two coprime integers,
and it is
also a semigroup generated by the interval
$\{4,5\}$.

In Table~\ref{table:hermite} we include,
for each integer from $0$ to $16$,
the values $\lambda_i$, $\nu_i$ and $d_{ORD}(C_i)$.
Notice that in this case the conductor is $12$,
the dominant is $10$ and the subconductor is $8$.
We do not give the values
in the table for
$i> \lambda^{-1}(2c-1)-1=16$
because
$d_{ORD}(C_i)=\nu_{i+1}$ for all $i\geq \lambda^{-1}(2c-1)-1$.
We can check that, as
follows from Theorem~\ref{theorem: ultim punt decreixement per acute semigroups}
and Lemma~\ref{lemma: qui es el minim dels dos},
$\lambda^{-1}(c+c'-2)=12$ is the largest integer $m$
with $\nu_m>\nu_{m+1}$ and so
the smallest integer for which
$d_{ORD}(C_i)=\nu_{i+1}$
for all $i\geq m$.
Notice also that, as
pointed out in Lemma~\ref{lemma: qui es el minim dels dos},
$c+c'-2=2c-2-\lambda_1$.

Furthermore, in this example
there are $64$ rational points
on the curve different from $P_\infty$ 
and the map $\varphi$
evaluating the functions of $A$ at these $64$ points
satisfies
that the words $\varphi(f_0),\dots,\varphi(f_{57})$
are linearly independent whereas $\varphi(f_{58})$
is linearly dependent to the previous ones.
So, $d_{ORD}^{P_1,\dots,P_n}(C_i)=d_{ORD}(C_i)$
for all $i\leq 56$.
\end{example}

\begin{table}
\caption{Hermitian curve}
\label{table:hermite}
\centering
\begin{tabular}{|c|c|c|c|}
\hline
$i$ & $\lambda_i$ & $\nu_i$ & $d_{ORD}(C_i)$ \\ \hline
$0$ & $0$ & $1$ & $2$ \\
$1$ & $4$ & $2$ & $2$ \\
$2$ & $5$ & $2$ & $3$ \\
$3$ & $8$ & $3$ & $3$ \\
$4$ & $9$ & $4$ & $3$ \\
$5$ & $10$ & $3$ & $4$ \\
$6$ & $12$ & $4$ & $4$ \\
$7$ & $13$ & $6$ & $4$ \\
$8$ & $14$ & $6$ & $4$ \\
$9$ & $15$ & $4$ & $5$ \\
$10$ & $16$ & $5$ & $8$ \\
$11$ & $17$ & $8$ & $8$ \\
$12$ & $18$ & $9$ & $8$ \\
$13$ & $19$ & $8$ & $9$ \\
$14$ & $20$ & $9$ & $10$ \\
$15$ & $21$ & $10$ & $12$ \\
$16$ & $22$ & $12$ & $12$ \\
\hline \end{tabular}
\end{table}

\begin{example}
Let us consider now the
semigroup of the fifth code associated to the second tower of
Garcia and Stichtenoth over ${\mathbb F}_{4}$.
%\index{Garcia-Stichtenoth towers!semigroup}
As noticed
in Example~\ref{example:2nd GS es Arf},
this is an Arf numerical semigroup.
We set in Table~\ref{table:garciastichtenoth}
the values $\lambda_i$, $\nu_i$ and $d_{ORD}(C_i)$
for each integer from $0$ to $25$.
In this case the conductor is $24$,
the dominant is $20$ and the subconductor is $20$.
As before, we do not give the values for
$i> \lambda^{-1}(2c-1)-1=25$.
We can check that, as
follows from Theorem~\ref{theorem: ultim punt decreixement per acute semigroups}
and Lemma~\ref{lemma: qui es el minim dels dos},
$\lambda^{-1}(2d)=19$ is the largest integer $m$
with $\nu_m>\nu_{m+1}$ and so,
the smallest integer for which
$d_{ORD}(C_i)=\nu_{i+1}$
for all $i\geq m$.
\end{example}

\begin{table}
\caption{Garcia-Stichtenoth tower}
\label{table:garciastichtenoth}
\centering
\begin{tabular}{|c|c|c|c|}
\hline
$i$ & $\lambda_i$ & $\nu_i$ & $d_{ORD}(C_i)$ \\ \hline
$0$ & $0$ & $1$ & $2$ \\
$1$ & $16$ & $2$ & $2$ \\
$2$ & $20$ & $2$ & $2$ \\
$3$ & $24$ & $2$ & $2$ \\
$4$ & $25$ & $2$ & $2$ \\
$5$ & $26$ & $2$ & $2$ \\
$6$ & $27$ & $2$ & $2$ \\
$7$ & $28$ & $2$ & $2$ \\
$8$ & $29$ & $2$ & $2$ \\
$9$ & $30$ & $2$ & $2$ \\
$10$ & $31$ & $2$ & $2$ \\
$11$ & $32$ & $3$ & $2$ \\
$12$ & $33$ & $2$ & $2$ \\
$13$ & $34$ & $2$ & $2$ \\
$14$ & $35$ & $2$ & $2$ \\
$15$ & $36$ & $4$ & $2$ \\
$16$ & $37$ & $2$ & $2$ \\
$17$ & $38$ & $2$ & $2$ \\
$18$ & $39$ & $2$ & $4$ \\
$19$ & $40$ & $5$ & $4$ \\
$20$ & $41$ & $4$ & $4$ \\
$21$ & $42$ & $4$ & $4$ \\
$22$ & $43$ & $4$ & $6$ \\
$23$ & $44$ & $6$ & $6$ \\
$24$ & $45$ & $6$ & $6$ \\
$25$ & $46$ & $6$ & $6$ \\
\hline \end{tabular}
\end{table}

Munuera and Torres in \cite{MuTo2008} 
and Oneto and Tamone in \cite{OnTa2008} 
proved that
for {\em any} numerical semigroup 
$m\leq\min\{c+c'-2-g,2d-g\}.$
Notice that for acute semigroups 
this inequality is an equality.

Munuera and Torres in \cite{MuTo2008} 
introduced the definition of near-acute semigroups.
They proved that the formula $m=\min\{c+c'-2-g,2d-g\}$
not only applies for acute semigroups but also for near-acute semigroups.
Next we give the definition of near-acute semigroups.

\begin{definition}
A numerical semigroup with conductor $c$, dominant $d$ and subdominant $d'$
is said to be a
\alert{near-acute semigroup} if
either $c-d\leq d-d'$ or $2d-c+1\not\in\Lambda$.
\end{definition}

Oneto and Tamone in \cite{OnTa2008} proved that 
$m=\min\{c+c'-2-g,2d-g\}$ if and {\em only if} $c+c'-2\leq 2d$ or $2d-c+1\not\in\Lambda$.
Let us see next that these conditions 
in Oneto and Tamone's result
are equivalent to having a near-acute semigroup.

\begin{lemma}
For a numerical semigroup the following are equivalent
\begin{enumerate}
\item $c-d\leq d-d'$ or $2d-c+1\not\in\Lambda$,
\item $c+c'-2\leq 2d$ or $2d-c+1\not\in\Lambda$.
\end{enumerate}
\end{lemma}

\begin{proof}
Let us see first that (1) implies (2).
If $2d-c+1\not\in\Lambda$ then it is obvious.
Otherwise the condition $c-d\leq d-d'$ is equivalent to
$d'\leq 2d-c$ 
which, together with 
$2d-c+1\in\Lambda$ implies 
$c'\leq 2d-c+1$ by definition of $c'$.
This in turn implies that $c+c'-2<c+c'-1\leq 2d$.

To see that (1) is a consequence of (2)
notice that by definition, $d'\leq c'-2$. Then,
if $c+c'-2\leq 2d$, we have
$d-d'\geq d-c'+2\geq c-d$.
\end{proof}

From all these results one concludes the next theorem.

\begin{theorem}
\begin{enumerate}
\item
For
{\em any} numerical semigroup 
$m\leq\min\{c+c'-2-g,2d-g\}.$
\item $m=\min\{c+c'-2-g,2d-g\}$
if and only if the corresponding numerical semigroup is near-acute.
\end{enumerate}
\end{theorem}

In \cite{OnTa2009} Oneto and Tamone give further results on $m$ and 
in \cite{OnTa2010} the same authors conjecture that
for any numerical semigroup,
$$\lambda_m\geq c+d-\lambda_1.$$

\subsubsection{The $\nu$ sequence and Feng-Rao improved codes}

The one-point codes whose set $W$ of parity checks is selected 
so that the orders outside $W$
satisfy the hypothesis of Lemma~\ref{lemma: condicio mes feble} 
and $W$ is minimal with this property are called Feng-Rao improved codes.
They were defined in \cite{FeRa:improved,HoLiPe:agc}.

\begin{definition}
Given a rational point $P$ of an
algebraic smooth curve ${\mathcal X}_F$ defined over ${\mathbb F}_q$
with Weierstrass semigroup $\Lambda$ and sequence $\nu$
with associated basis $z_0,z_1,\dots$
and given $n$ other different points $P_1,\dots,P_n$ of ${\mathcal X}_F$,
the associated \alert{Feng-Rao improved code}
guaranteeing correction of $t$ errors
is defined as $$C_{\tilde R(t)}=<(z_i(P_1),\dots,z_i(P_n)):i\in \tilde R(t)>^\perp,$$
where $$\tilde R(t)=\{i\in{\mathbb N}_0:\nu_i<2t+1\}.$$
\end{definition}

\subsubsection{On the improvement of the Feng-Rao improved codes}
\label{section: nu no decreasing implica ordinary}

Feng-Rao improved codes will actually give an improvement with respect to
classical codes only if $\nu_i$ is decreasing at some $i$.
We next study this condition.

\begin{lemma}
\label{lemma: nui ordinary}
If $\Lambda$ is an ordinary numerical semigroup with enumeration $\lambda$
then
$$\nu_i=\trescasos{1}{if $i=0$,}{2}{if $1\leq i\leq \lambda_1$,}
{i-\lambda_1+2}{if $i>\lambda_1$.}$$
\bigskip
\end{lemma}

\begin{proof}It is obvious that $\nu_0=1$
and that $\nu_i=2$ whenever
$0<\lambda_i<2\lambda_1$.
So, since $2\lambda_1=\lambda_{\lambda_1+1}$,
we have that $\nu_i=2$ for all $1\leq i\leq \lambda_1$.
Finally, if $\lambda_i\geq 2\lambda_1$ then
all non-gaps up to $\lambda_i-\lambda_1$
are in $N_i$ as well as $\lambda_i$, and
none of the remaining non-gaps are in $N_i$.
Now, if the genus of $\Lambda$ is $g$, then
$\nu_i=\lambda_i-\lambda_1+2-g$
and
$\lambda_i=i+g$.
So,
$\nu_i=i-\lambda_1+2$.
\end{proof}

As a consequence of
Lemma~\ref{lemma: nui ordinary},
the $\nu$ sequence is non-decreasing if $\Lambda$
is an ordinary numerical semigroup.
We will see in this section that
ordinary numerical semigroups are in fact the only semigroups for
which the $\nu$ sequence  is non-decreasing.

\begin{lemma}
\label{lemma: nu non-decreasing implica Arf}
Suppose that for the semigroup $\Lambda$ the $\nu$ sequence
is non-de\-crea\-sing. Then $\Lambda$ is Arf.
\end{lemma}

\begin{proof}
Let $\lambda$ be the enumeration of $\Lambda$.
Let us see by induction that, for any non-negative integer $i$,
\begin{description}
\item[{\bf (i)}]
$N_{\lambda^{-1}(2\lambda_i)}=
  \{j\in{\mathbb N}_0: j\leq i\}\sqcup\{\lambda^{-1}(2\lambda_i-\lambda_j): 0\leq j<i\}$,
where $\sqcup$ means the union
of disjoint sets.
\item[{\bf (ii)}]
$N_{\lambda^{-1}(\lambda_i+\lambda_{i+1})}=
\{j\in{\mathbb N}_0: j\leq i\}\sqcup\{\lambda^{-1}(\lambda_i+\lambda_{i+1}-\lambda_j): 0\leq j\leq i\}$.
\end{description}
Notice that if {\bf (i)}
is satisfied for all $i$,
then
$\{j\in{\mathbb N}_0: j\leq i\}\subseteq N_{\lambda^{-1}(2\lambda_i)}$
for all $i$, and hence
by
Lemma~\ref{lemma: charactarf 2j-i}
$\Lambda$ is Arf.

It is obvious
that both {\bf (i)} and {\bf (ii)}
are satisfied
for the case $i=0$.

Suppose $i>0$.
By the induction hypothesis,
$\nu_{\lambda^{-1}(\lambda_{i-1}+\lambda_i)}=2i$.
Now, since $(\nu_i)$ is not decreasing
and $2\lambda_i>\lambda_{i-1}+\lambda_i$,
we have
$\nu_{\lambda^{-1}(2\lambda_i)}\geq 2i$.
On the other hand,
if $j,k\in{\mathbb N}_0$ are such that $j\leq k$ and $\lambda_j+\lambda_k=2\lambda_i$
then $\lambda_j\leq\lambda_i$ and $\lambda_k\geq \lambda_i$.
So,
$\lambda(N_{\lambda^{-1}(2\lambda_i)})\subseteq
\{\lambda_j: 0\leq j\leq i\}\sqcup\{2\lambda_i-\lambda_j: 0\leq j<i\}$
and hence
$\nu_{\lambda^{-1}(2\lambda_i)}\geq 2i$ if and
only if
$N_{\lambda^{-1}(2\lambda_i)}=
\{j\in{\mathbb N}_0: j\leq i\}\sqcup\{\lambda^{-1}(2\lambda_i-\lambda_j): 0\leq j<i\}$.
This proves {\bf (i)}.

Finally, {\bf (i)} implies
$\nu_{\lambda^{-1}(2\lambda_i)}=2i+1$
and {\bf (ii)} follows by an analogous argumentation.
\end{proof}

\begin{theorem}
\label{theorem: non-decreasing nu implica ordinary}
The only numerical semigroups for which the $\nu$ sequence
is non-de\-crea\-sing are ordinary numerical semigroups.
%\index{semigroup!ordinary}
\end{theorem}

\begin{proof}
It is a consequence of Lemma~\ref{lemma: nu non-decreasing implica Arf},
Lemma~\ref{lemma: exemples acute},
Theorem~\ref{theorem: ultim punt decreixement per acute semigroups}
and Lemma~\ref{lemma: nui ordinary}.
\end{proof}

\begin{corollary}\label{corollary: nu increasing implica semigrup trivial}
The only numerical semigroup for which
the $\nu$ sequence 
is strictly increasing is the trivial numerical semigroup.
\end{corollary}

\begin{proof}
It is a consequence of
Theorem~\ref{theorem: non-decreasing nu implica ordinary}
and Lemma~\ref{lemma: nui ordinary}.
\end{proof}

As a consequence of 
Theorem~\ref{theorem: non-decreasing nu implica ordinary}
we can show that the only numerical semigroups
for which the associated classical codes
are not improved 
by the Feng-Rao improved codes,
at least for one value of $t$, are
ordinary semigroups.

\begin{corollary}
Given a numerical semigroup $\Lambda$ define
$m(\delta)=\max\{i\in{\mathbb N}_0: \nu_i<\delta\}$.
There exists at least one value of $\delta$
for which $\{i\in{\mathbb N}_0:\nu_i<\delta\}\subsetneq\{i\in{\mathbb N}_0: i\leq m(\delta)\}$
if and only if $\Lambda$ is non-ordinary.
\end{corollary}

\subsection{Generic errors and the $\tau$ sequence}
% and codes guaranteeing correction of generic errors}

All the results in these sections are based on \cite{BrOS:2006_AAECC,BrOS:duality,Bras:2009_DCC}.
Correction of generic errors has already been considered in \cite{Pellikaan92,O'Sullivan:hermite-beyond,JensenNielsenHoholdt}.

\subsubsection{Generic errors}

\begin{definition}
The points $P_{i_1},\dots,P_{i_t}$ ($P_{i_j}\neq P$) 
are \alert{generically distributed} if no non-zero function
generated by $z_0,\dots,z_{t-1}$ vanishes in all of them.
In the context of one-point codes, \alert{generic errors} are those 
errors whose non-zero positions correspond to generically distributed 
points. Equivalently, $e$ is generic if and only if $\Delta_e=\Delta_t:=\{0,\dots,t-1\}.$
\end{definition}

Generic errors of weight $t$ can be a very large proportion of all possible errors of weight $t$ \cite{Hansen}.
Thus, by restricting the errors to be corrected to generic errors 
the decoding requirements become weaker and we are still able to correct almost all errors.
In some of these references generic errors are called \alert{independent errors}.

\begin{example}[Generic sets of points in ${\mathcal H}_q$]
Recall that the Hermitian curve 
from Example~\ref{example:hermite}.
It is defined over ${\mathbb F}_{q^2}$ and its 
affine equation is $x^{q+1}=y^q+y$.

The unique point at infinity is $P_\infty=(0:1:0)$.
If $b\in{\mathbb F}_q$ then $b^q+b=Tr(b)=0$ and the unique affine
point with $y=b$ is $(0,b)$. There are a total of $q$ such points.
If $b\in{\mathbb F}_{q^2}\setminus{\mathbb F}_q$ then $b^q+b=Tr(b)\in{\mathbb F}_q\setminus\{0\}$ and there are $q+1$ solutions of $x^{q+1}=b^q+b$, so, there are 
$q+1$ different affine
points with $y=b$. There are a total of $(q^2-q)(q+1)$ such points.
The total number of affine points is then 
$q+(q^2-q)(q+1)=q^3$.

If we distinguish the point $P_\infty$, 
we can take $z_0=1, z_1=x$, $z_2=y$,
$z_3=x^2$, $z_4=xy$,
$z_5=y^2$\dots

Non-generic sets of two points are pairs of points satisfying 
$x^{q+1}=y^q+y$ and simultaneously vanishing at $f=z_1+a z_0=x+a$ for some $a\in{\mathbb F}_{q^2}$.
The expression $x+a$ represents a line with $q$ points.
There are $q^2$ such lines.
There are a total of $q^2\binom{q}{2}$ 
pairs of colinear points over lines of the form $x+a$ and so 
$q^2\binom{q}{2}$ non-generic errors.

Consequently, the portion of non-generic errors of weight $2$ is $\frac{q^2\binom{q}{2}}{\binom{q^3}{2}}
%=\frac{q^3(q-1)}{q^3(q^3-1)}=\frac{q-1}{q^3-1}
=\frac{1}{q^2+q+1}$.

A set of three points is non-generic if the points satisfy  
$x^{q+1}=y^q+y$ and simultaneously vanish at 
$f=z_1+a z_0=x+a$ for some $a\in{\mathbb F}_{q^2}$ or at
$f=z_2+a z_1+b z_0=y+ax+b$ for some $a,b\in{\mathbb F}_{q^2}$.

The expression $x+a$ represents a line (which we call of type 1)
with $q$ points. There are $q^2$ lines of type 1.

The line $y+ax+b$ is called of type 2 if $a^{q+1}= b^q+b$
and of type 3 otherwise.
There are $q^3$ lines of type 2 and $q^4-q^3$ lines of type 3.

Lines of type 2 have only one point. Indeed,
a point on ${\mathcal H}_q$ 
and on the line
 $y+ax+b$ 
must satisfy $x^{q+1}=(-ax-b)^q+(-ax-b)=
-(ax)^q-ax-a^{q+1}$.
Notice that $(x+a^q)^{q+1}=(x+a^q)^q(x+a^q)=(x^q+a)(x+a^q)=x^{q+1}+x^qa^q+ax+a^{q+1}$.
So, $x=-a^q$ is the unique solution to $x^{q+1}=-(ax)^q-ax-a^{q+1}$
and so the unique point of 
${\mathcal H}_q$ on the line 
 $y+ax+b$ is $(-a^q,a^{q+1}-b)$.

%Lines of type 2 have $q+1$ points.
%Indeed,
%a point on ${\mathcal H}_q$ and on the line
% $y+ax+b$ must satisfy $x^{q+1}=(-ax-b)^q+(-ax-b)=
%-(ax)^q-b^q-ax-b$.
%Let us see that the solutions to 
%$x^{q+1}+(ax)^q+b^q+ax+b=0$
%are all different.
%Let us call $g=x^{q+1}+(ax)^q+b^q+ax+b$.
%This can be done by proving that $g$ and its derivative have no roots in comon.
%The derivative is $g'=x^q+a$ and has roots in common with
%$g$ only if $-a^{q+1}+b^q+b=0$.

%Lines of type 3 have only one point.
%This follows by a counting argument. There are a total of $\binom{q^3}{2}$ 
%pairs of affine points. Each pair meets only in one line.
%The number of pairs sharing lines of type 1 is $q^2\binom{q}{2}$
%while the number of pairs sharing lines of type 2 is $q^3(q-1)\binom{q+1}{2}$.
%Since $q^2\binom{q}{2}+q^3(q-1)\binom{q+1}{2}=\binom{q^3}{2}$,
%there can be no pair of points meeting in any of the lines of type 3. So 
%the lines of type 3
%may have at most one point each and in fact they have one point each.

Lines of type 3 have $q+1$ points.
This follows by a counting argument. 
On one hand, as seen before,
a point on ${\mathcal H}_q$ 
and on the line
 $y+ax+b$ 
must satisfy $x^{q+1}=-(ax)^q-ax-b^q-b$.
There are at most $q+1$ different values of $x$ satisfying this equation 
and so at most $q+1$ different points of 
${\mathcal H}_q$ on the line
$y+ax+b$. 
On the other hand
there are a total of $\binom{q^3}{2}$ 
pairs of affine points. Each pair meets only in one line.
The number of pairs sharing lines of type 1 is $q^2\binom{q}{2}$, 
the number of pairs sharing lines of type 2 is $0$ and
the number of pairs sharing lines of type 3 is at most $q^3(q-1)\binom{q+1}{2}$, with equality only if all lines of type 3 have $q+1$ points.
Since $q^2\binom{q}{2}+q^3(q-1)\binom{q+1}{2}=\binom{q^3}{2}$,
we deduce that all the lines of type 3 must have $q+1$ points.

In total there are $q^2\binom{q}{3}$ 
sets of three points 
sharing a line of type 1 and 
$(q^4-q^3)\binom{q+1}{3}$
sets of three points 
sharing a line of type 3.

The portion of non-generic errors of weight $3$ is then  
$\frac{q^2\binom{q}{3}+q^3(q-1)\binom{q+1}{3}}{\binom{q^3}{3}}=
\frac{1}{q^2+q+1}$.
\end{example}

\subsubsection{Conditions for correcting generic errors}

In the next lemma we find conditions guaranteeing the majority voting step 
for generic errors.
It is a reformulation of results that appeared in 
\cite{O'Sullivan:manuscript,BrOS:2006_AAECC,Bras:2009_DCC}
%\cite{BrOS:2006_AAECC,2009_DCC}

\begin{lemma}
\label{lemma:condicionspercalcularsindromes_generic}
Let $\Sigma_t={\mathbb N}_0\setminus\Delta_t=t+{\mathbb N}_0$.
The following conditions are equivalent.
\begin{enumerate}
\item $\nu_k>2\#(N_k\cap\Delta_t)$,
\item $k\in\Sigma_t\oplus\Sigma_t$,
\item $\tau_k\geq t$.
\end{enumerate}
\end{lemma}

\begin{proof}
Let $A=\{i\in N_k: i, k\ominus i\in \Delta_t\}$, 
%$B=\{i\in N_k: i\in\Delta_e, k\ominus i\not\in \Delta_e\}$, 
%$C=\{i\in N_k: i\not\in\Delta_e, k\ominus i\in \Delta_e\}$, 
$D=\{i\in N_k: i, k\ominus i\in \Sigma_t\}$.
By an argument analogous to that in the proof of 
Lemma~\ref{lemma:condicionspercalcularsindromes},
$\nu_k>2\#(N_k\cap\Delta_e)$ is equivalent to 
$\#D>\#A$.
If this inequality is satisfied then $\#D>0$ and so 
$k\in\Sigma_t\oplus\Sigma_t$.
On the other hand, $\min \Sigma_t\oplus\Sigma_t = t\oplus t>(t-1)\oplus(t-1)=\max \Delta_t\oplus\Delta_t$. So, $\Sigma_t\oplus\Sigma_t\cap \Delta_t\oplus\Delta_t=\emptyset$ and, if $k\in \Sigma_t\oplus\Sigma_t$ then $k\not\in\Delta_t\oplus\Delta_t$ and so $\#A=0$ implying $\#D>\#A$.

The equivalence of $k\in\Sigma_t\oplus\Sigma_t$ and $\tau_k\geq t$ is straightforward.
\end{proof}

The one-point codes whose set $W$ of parity checks is selected 
so that the orders outside $W$
satisfy the hypothesis of Lemma~\ref{lemma:condicionspercalcularsindromes_generic}
and $W$ is minimal with this property are called improved codes correcting generic errors.
They were defined in \cite{Bras:2003_AAECC,BrOS:2006_AAECC}.

\begin{definition}
Given a rational point $P$ of an
algebraic smooth curve ${\mathcal X}_F$ defined over ${\mathbb F}_q$
with Weierstrass semigroup $\Lambda$ and sequence $\nu$
with associated basis $z_0,z_1,\dots$
and given $n$ other different points $P_1,\dots,P_n$ of ${\mathcal X}_F$,
the associated \alert{improved code}
guaranteeing correction of $t$ generic errors
is defined as $$C_{\tilde R^*(t)}=<(z_i(P_1),\dots,z_i(P_n)):i\in \tilde R^*(t)>^\perp,$$
where $$\tilde R^*(t)=\{i\in{\mathbb N}_0:\tau_i<t\}.$$
\end{definition}

\subsubsection{Comparison of improved codes and classical codes correcting generic errors}
\label{section:comparewithclassical}

Classical evaluation codes are those codes for which the set of
parity checks corresponds to all the elements up to a given order.
Thus, the classical evaluation code with maximum dimension correcting
$t$ generic errors is
defined by the set of checks
$R^*(t)=\{i\in{\mathbb N}_0:i\leq m(t)\}$
where $m(t)=\max\{i\in{\mathbb N}_0:\tau_i<t\}$.
Then, by studying the 
monotonicity
%increasingness
of the $\tau$ sequence we can
compare $\widetilde{R}^*(t)$ and $R^*(t)$ and the associated codes.

%Let us introduce some new notations
%that we will use from now on.
%Given a numerical semigroup $\Lambda$,
%the elements in the complement ${\mathbb N}_0\setminus \Lambda$
%are called the \alert{gaps} of the numerical semigroup
%and $\#{\mathbb N}_0\setminus\Lambda$ is its \alert{genus}.
%The \alert{conductor} of $\Lambda$
%is the unique integer $c\in\Lambda$ such that $c-1\not\in\Lambda$
%and $c+{\mathbb N}_0\subseteq \Lambda$.
%We say that a numerical semigroup is \alert{ordinary} if it is equal to
%$\{0\}\cup\{i\in{\mathbb N}_0:i\geq c\}$ for some non-negative integer $c$.
%If moreover $c=0$ then the numerical semigroup
%is ${\mathbb N}_0$ and in this case we say that it is the \alert{trivial} semigro%up.
%For a non-trivial semigroup, the non-gap previous to the conductor is the
%\alert{dominant}.
%Usually we use $g$, $c$, $d$ for the genus, the conductor and the dominant of a
%numerical semigroup. Notice that $\lambda_{c-g}=c$ and $\lambda_{c-g-1}=d$.

It is easy to check that for the trivial numerical semigroup
one has $\tau_{2i}=\tau_{2i+1}=i$
for all $i\in{\mathbb N}_0$. That is, the $\tau$ sequence is
$$0,0,1,1,2,2,3,3,4,4,5,5,\dots$$
The next lemma determines the
$\tau$ sequence of all non-trivial ordinary semigroups.

\begin{lemma}
\label{lemma:tauordinary}
The non-trivial ordinary numerical semigroup with conductor $c$
has $\tau$ sequence given by
$$\tau_i=\left\{\begin{array}{ll}0&\mbox{if }i\leq c\\
\left\lfloor\frac{i-c+1}{2}\right\rfloor&\mbox{if }i> c\\
\end{array}\right.$$
\end{lemma}

\begin{proof}
Suppose that the numerical semigroup has enumeration $\lambda$.
On one hand,
$\lambda_1,\dots,\lambda_{c}$ are all generators
and thus $\tau_i=0$ for $i\leq c$.
For $i>c$, $\lambda_{i}=c+i-1\geq 2c$. So,
if $\lambda_i$ is even
(which is equivalent to $c+i$ being
odd)
then
$\tau_i=\lambda^{-1}(\frac{\lambda_{i}}{2})=\frac{c+i-1}{2}-c+1=\frac{i-c+1}{2}=
\lfloor\frac{i-c+1}{2}\rfloor$.
If $\lambda_i$ is odd
(which is equivalent to either both $c$ and $i$ being
even or being odd)
then
$\tau_i=\lambda^{-1}(\frac{\lambda_{i}-1}{2})=\frac{c+i-2}{2}-c+1=\frac{i-c}{2}=
\lfloor\frac{i-c+1}{2}\rfloor$.
\end{proof}

\begin{remark}
\label{remark:curiousbehavior}
The formula in Lemma~\ref{lemma:tauordinary} 
can be reformulated as
$\tau_j=0$ for all $j\leq c$ and, for all $i\geq 0$,
$\tau_{c+2i+1}=\tau_{c+2i+2}=i+1$.
\end{remark}

The next lemma gives, for non-ordinary semigroups,
the smallest index $m$ for which
$\tau$ is non-decreasing from $\tau_m$ on.
We will use the notation $\lfloor a\rfloor_\Lambda$
to denote the semigroup floor of a non-negative integer
$a$, that is, the largest non-gap of $\Lambda$
which is at most $a$.

\begin{lemma}
\label{lemma:trobam}
Let $\Lambda$ be a non-ordinary semigroup with dominant $d$
and let $m=\lambda^{-1}(2d)$, then
\begin{enumerate}
\item $\tau_m=c-g-1>\tau_{m+1}$,
\item $\tau_i<c-g-1$ for all $i<m$,
\item $\tau_{i}\leq\tau_{i+1}$ for all $i>m.$
\end{enumerate}
\end{lemma}

\begin{proof}
For statement 1 notice that both $2d$ and $2d+1$ belong to $\Lambda$
because they must be larger than
the conductor. Furthermore,
$\tau_{\lambda^{-1}(2d)}=\lambda^{-1}(d)=c-g-1$
while $\tau_{\lambda^{-1}(2d+1)}=\tau_{\lambda^{-1}(2d)+1}<\lambda^{-1}(d)$
because $d+1\not\in\Lambda$.

Statement 2 follows from the fact that if $\lambda_i<2d$
then $\tau_i<\lambda^{-1}(d)=c-g-1$.

For statement 3 suppose that $i>m$.
Notice that $2d$ is the largest non-gap that can be written as a
sum of two non-gaps both of them smaller than the conductor $c$.
Then if
$j\leq k\leq i$ and $\lambda_j+\lambda_k=\lambda_i$
it must be $\lambda_k\geq c$ and so
$\tau_i=\lambda^{-1}(\lfloor\lambda_i-c\rfloor_\Lambda)$.
Since both $\lambda^{-1}$ and $\lfloor\cdot\rfloor_\Lambda$
are non-decreasing, so is $\tau_i$ for $i>m$.
\end{proof}

\begin{corollary}
\label{corollary:charordinary}
The only numerical semigroups for which the $\tau$ sequence is non-decreasing
are ordinary semigroups.
\end{corollary}

A direct consequence of Corollary~\ref{corollary:charordinary} is that
the classical code determined by $R^*(t)$ is always worse
than the improved code determined by $\widetilde{R}^*(t)$
at least for one value of $t$ unless the corresponding
numerical semigroup is ordinary.
%An analogous result is proved in \cite[Corollary 7.5]{Bras:2004_IEEE_Acute}
%for Feng-Rao improved codes. In this case, Feng-Rao improved codes
%actually improve classical codes for any numerical semigroup except for
%ordinary semigroups.
From Lemma~\ref{lemma:trobam} we can derive that $\widetilde{R}^*(t)$
and $R^*(t)$ coincide from a certain point and we can find this point.
We summarize the results of this section in the next Corollary.

\begin{corollary}
\begin{enumerate}
\item $\widetilde{R}^*(t)\subseteq R^*(t)$ for all $t\in{\mathbb N}_0$.
\item $\widetilde{R}^*(t)= R^*(t)$ for all $t\geq c-g$.
\item $\widetilde{R}^*(t)= R^*(t)$ for all $t\in{\mathbb N}_0$
if and only if the associated numerical semigroup is ordinary.
\end{enumerate}
\end{corollary}

\begin{proof}
Statement 1 is a consequence of the definition of $R^*(t)$.
Statement 2 is clear if the associated semigroup is ordinary.
Otherwise it follows from the fact proved in Lemma~\ref{lemma:trobam}
that the largest value of $\tau_i$
before it starts being non-decreasing is precisely $c-g-1$
and that before that all values of $\tau_i$ are smaller than $c-g-1$.
Statement 3 is a consequence of Corollary~\ref{corollary:charordinary}.
\end{proof}

\subsubsection{Comparison of improved codes correcting generic errors and Feng--Rao improved codes}
%\section{Characterization of Arf semigroups by
%means of the $\tau$ sequence}
\label{section:comparewithFR}

In next theorem we compare $\tau_i$ with $\lfloor\frac{\nu_i-1}{2}\rfloor$
and this will give a new characterization of Arf semigroups.
Recall that $t\leq \lfloor\frac{\nu_i-1}{2}\rfloor$ 
guarantees the computation of syndromes of order $i$ when performing majority voting
(Lemma~\ref{lemma:condicionspercalcularsindromes}).

\begin{theorem}
\label{t:nutau}
Let $\Lambda$ be a numerical semigroup with conductor $c$, genus $g$,
and associated sequences $\tau$ and $\nu$. Then
\begin{enumerate}
\item
$\tau_i\geq\lfloor\frac{\nu_i-1}{2}\rfloor$ for all $i\in{\mathbb N}_0$,
\item
$\tau_i = \lfloor\frac{\nu_i-1}{2}\rfloor$ for all $i\geq 2c-g-1$,
\item
$\tau_i = \lfloor\frac{\nu_i-1}{2}\rfloor$ for all $i\in{\mathbb N}_0$
if and only if $\Lambda$ is Arf.
\end{enumerate}
\end{theorem}

\begin{proof}
Let $\lambda$ be the enumeration of $\Lambda$.
\begin{enumerate}
\item 
%For $i\in{\mathbb N}_0$ let $N_i=\{j\in{\mathbb N}_0: \lambda_i-\lambda_j\in\Lambda\}$ and 
Suppose that the elements in $N_i$ are ordered $N_{i,0}< N_{i,1}< N_{i,2}
<\dots< N_{i,\nu_i-1}$. On one hand $\tau_i=N_{i,\lfloor\frac{\nu_i-1}{2}\rfloor}$.
On the other hand $N_{i,j}\geq j$ and this finishes the proof of the first statement.

\item
The result is obvious for the trivial semigroup. Thus we can assume that
$c\geq g+1$.
Notice that
$\tau_i=N_{i,\lfloor\frac{\nu_i-1}{2}\rfloor}=\lfloor\frac{\nu_i-1}{2}\rfloor$
if and only if all integers less than or equal to $\lfloor\frac{\nu_i-1}{2}\rfloor$
belong to $N_i$.
Now let us prove that if $i\geq 2c-g-1$ then
all integers less than or equal to $\lfloor\frac{\nu_i-1}{2}\rfloor$
belong to $N_i$.
Indeed, if $j\leq\lfloor\frac{\nu_i-1}{2}\rfloor$ then
$\lambda_j\leq \lambda_i/2$
and $\lambda_i-\lambda_j\geq \lambda_i-\lambda_i/2=\lambda_i/2\geq c-1/2$.
Since $\lambda_i-\lambda_j\in{\mathbb N}_0$ this means that
$\lambda_i-\lambda_j\geq c$ and so $\lambda_i-\lambda_j\in\Lambda$.

\item
Suppose that $\Lambda$ is Arf.
We want to show that for any non-negative integer $i$, all non-negative integers
less than or equal to $\lfloor\frac{\nu_i-1}{2}\rfloor$
belong to $N_i$.
By definition of $\tau_i$ there exists $k$
with $\tau_i\leq k\leq i$ and $\lambda_{\tau_i}+\lambda_{k}=\lambda_i$.
Now, if $j$ is a non-negative integer with $j\leq \lfloor\frac{\nu_i-1}{2}\rfloor$,
by statement 1 it also satisfies $j\leq \tau_i$. Then
$\lambda_i-\lambda_j=\lambda_{\tau_i}+\lambda_k-\lambda_j\in\Lambda$ by the
Arf property, and so $j\in N_i$.

On the other hand, suppose that
$\tau_i=\lfloor\frac{\nu_i-1}{2}\rfloor$
for all non-negative integer $i$.
This means that all integers less than or equal to $\tau_r$
belong to $N_r$ for any non-negative integer $r$.
If $i\geq j\geq k$ then $\tau_{\lambda^{-1}(\lambda_i+\lambda_j)}\geq j\geq k$
and by hypothesis $k\in N_{\lambda^{-1}(\lambda_i+\lambda_j)}$, which means that
$\lambda_i+\lambda_j-\lambda_k\in\Lambda$. This implies that $\Lambda$ is Arf.
\end{enumerate}
\end{proof}

Statement 1) of Lemma~\ref{lemma:precharacterization}
for the case when $i>0$ is a direct consequence of
Theorem~\ref{t:nutau} and Lemma~\ref{lemma:nu}.

Finally, Theorem~\ref{t:nutau}
together with Lemma~\ref{lemma:precharacterization}
has the next corollary.
Different versions of this result appeared in
\cite{Bras:2003_AAECC,BrOS:2006_AAECC,Bras:2009_DCC}.
The importance of the result is that it shows that
the improved codes correcting generic errors 
do always require at most
as many checks as 
the Feng--Rao improved codes correcting any kind of errors.
It also
states conditions under which their redundancies are equal and
characterizes Arf semigroups as the unique semigroups
for which there is no improvement.

\begin{corollary}
\begin{enumerate}
\item $\widetilde{R}^*(t)\subseteq \widetilde{R}(t)$ for all $t\in{\mathbb N}_0$.
\item $\widetilde{R}^*(t)= \widetilde{R}(t)$ for all $t\geq c-g$.
\item $\widetilde{R}^*(t)= \widetilde{R}(t)$ for all $t\in{\mathbb N}_0$
if and only if the associated numerical semigroup is Arf.
\end{enumerate}
\end{corollary}

\begin{proof}
Statement 1. and 3. follow immediately from
Theorem~\ref{t:nutau}
and the fact that
$\widetilde{R}(t)=\{i\in{\mathbb N}_0:\left\lfloor\frac{\nu_i-1}{2}\right\rfloor<t\}$
and $\widetilde{R}^*(t)=\{i\in{\mathbb N}_0:\tau_i<t\}$.
For statement 2.,
we can use that for $i\geq 2c-g-1$,
$\tau_i=\lfloor\frac{\nu_i-1}{2}\rfloor$
(Theorem~\ref{t:nutau})
and that for $i\geq 2c-g-1$,
$\tau_i\geq c-g-1$ (Lemma~\ref{lemma:precharacterization}),
being $c-g-1$
the largest value of $\tau_j$
before it starts being non-decreasing (Lemma~\ref{lemma:trobam}).
\end{proof}

\section*{Further reading}

We tried to cite the specific bibliography related to each section 
within the text. Next we mention some more general references:
The book \cite{RoGa:llibre} has many results on numerical semigroups, including some of the problems presented in the first section of this chapter but also many others. The book \cite{RamirezAlfonsin} is also devoted to numerical semigroups from the perspective of the Frobenius' coin exchange problem.
Algebraic geometry codes have been widely explained in different books such as \cite{vLintvdGeer,Stichtenoth:AFFaC,Pretzel}. For one-point codes and also their relation with Weierstrass semigroups an important reference is the chapter \cite{HoLiPe:agc}.

\section*{Acknowledgments}

The author would like to thank 
Michael E. O'Sullivan, 
Ruud Pellikaan
and Pedro A. Garc\'\i a-S\'anchez
for many helpful discussions.
She would also like to thank all the coauthors of the papers involved in this chapter. 
They are, by order of appearence of the differet papers:
Michael E. O'Sullivan, Pedro A. Garc\'\i a-S\'anchez,
Anna de Mier, and Stanislav Bulygin.

This work was partly supported by
the Spanish Government through projects TIN2009-11689
``RIPUP'' and CONSOLIDER INGENIO 2010 CSD2007-00004 ``ARES'',
 and by the Government of Catalonia under grant 2009 SGR 1135.

%\bibliographystyle{plain}
%\bibliography{bib}

\end{document}